\documentclass[12pt, reqno]{amsart}
\usepackage{amsmath, amsthm, amscd, amsfonts, amssymb, graphicx, color}
\usepackage{subfig}
\usepackage[bookmarksnumbered, colorlinks, plainpages]{hyperref}
\hypersetup{colorlinks=true,linkcolor=red, anchorcolor=green, citecolor=cyan, urlcolor=red, filecolor=magenta, pdftoolbar=true}

\newtheorem{theorem}{Theorem}[section]

\newtheorem{proposition}[theorem]{Proposition}
\newtheorem{corollary}[theorem]{Corollary}
\theoremstyle{definition}
\newtheorem{definition}[theorem]{Definition}
\newtheorem{example}[theorem]{Example}

\newtheorem{Matlab Code}[theorem]{Matlab Code}

\theoremstyle{remark}
\newtheorem{remark}[theorem]{Remark}
\numberwithin{equation}{section}

\begin{document}

\setcounter{page}{1}

\title[Generalized Numerical Ranges; Max Algebras]{ results on the generalized numerical ranges in max algebra}

\author[N. Haj Aboutalebi, S. Fallat, A. Peperko, D. Taghizadeh,~\MakeLowercase{and} M. Zahraei]{Narges  Haj Aboutalebi$^1$, Shaun Fallat$^{2*}$, Aljo\v{s}a Peperko$^3$, Davod Taghizadeh$^4$,~\MakeLowercase {and}  Mohsen Zahraei$^{5}$}

\address{$^{1}$Department of Mathematics, Shahrood Branch, Islamic Azad University, Shahrood, Iran.}
\email{\textcolor[rgb]{0.00,0.00,0.84}{aboutalebi.n@yahoo.com}}

\address{$^{2}$Department of Mathematics and Statistics, University of Regina, Saskatchewan, Canada.}
\email{\textcolor[rgb]{0.00,0.00,0.84}{shaun.fallat@uregina.ca}}

\address{$^{3}$Faculty of Mechanical Engineering, University of Ljubljana, A\v{s}ker\v{c}eva 6, SI-1000 Ljubljana, Slovenia, and 
\newline{Institute of Mathematics, Physics and Mechanical, Jadranska 19, 1000 Ljubljana, Slovenia}}
\email{\textcolor[rgb]{0.00,0.00,0.84}{aljosa.peperko@fs.uni-lj.si}}

\address{$^4$Department of Mathematics, Ahvaz Branch, Islamic Azad University, Ahvaz, Iran.}
\email{\textcolor[rgb]{0.00,0.00,0.84}{t.davood1411@gmail.com}}

\address{$^{5*}$Department of Mathematics, Ahvaz Branch, Islamic Azad University, Ahvaz, Iran.}
\email{\textcolor[rgb]{0.00,0.00,0.84}{mzahraei326@gmail.com}}


\let\thefootnote\relax\footnote{Copyright 2018 by the Tusi Mathematical Research Group.}



\date{Received: xxxxxx; Revised: yyyyyy; Accepted: zzzzzz.
\newline \indent $^{*}$Corresponding author}

\begin{abstract}
Let $n$ and $k$ be two positive integers with $k\leq n$ and $C$ an $n \times n$ matrix with nonnegative entries. In this paper, the rank-$k$ numerical range in the max algebra setting is introduced and studied. The related notions of the max joint $k$-numerical range and the max joint $C$-numerical range of an entry-wise nonnegative matrix and an $m$-tuple of nonnegative matrices are also introduced.  Some interesting algebraic properties of these concepts are investigated.  
\end{abstract} \maketitle

\section{\textbf{Introduction  and preliminaries}}
The algebraic system max algebra and its isomorphic versions are widely used in information technology, discrete event dynamic systems, combinatorial optimization, mathematical physics, and even in certain aspects of DNA analysis (see, for example \cite{AG,B,GMW,GM,MPE,MP} and the references cited therein). Moreover, in the last two decades, max algebra techniques were used to solve certain linear algebra problems (see e.g. \cite{ED,GS}).

Let $\mathbb{R}_+$ denote the nonnegative number. For our purpose the max algebra in question consists of $\mathbb{R}_+$ with sum $a\oplus b=\max\{a,b\}$ and the standard product $ab$, where $a,b\in\mathbb{R}_+$.  Let $\mathbb{R}_+^n$ be the collection of all $n$-column vectors with entries in $\mathbb{R}_+$, and $M_{m\times n}(\mathbb{R}_+)$ denote the set of the all $m\times n$ matrices such that its entries belong to $\mathbb{R}_+$. For the case $m=n$, $M_{n\times n}(\mathbb{R}_+)$ is denoted by $M_n(\mathbb{R}_+)$. Let $A=(a_{ij})\in M_{m\times n}(\mathbb{R}_+)$ and $B=(b_{ij})\in M_{n\times l}(\mathbb{R}_+)$. The product of $A$ and $B$ in max algebra is denoted by $A\otimes B$, where $(A\otimes B)_{ij}=\max_{k=1,\dots ,n}a_{ik}b_{kj}$. Moreover, if $A,B\in M_{m\times  n}(\mathbb{R}_+)$, where $A=(a_{ij})$ and $B=(b_{ij})$, then the sum of $A$ and $B$ in this max algebra framework is denoted by $A\oplus B$, where $(A\oplus B)_{ij}=\max\{a_{ij},b_{ij}\}$ for $i=1,\dots ,m$  and $j=1,\dots , n$.  For 
$A\in M_{m\times  n}(\mathbb{R})$ and $x\in \mathbb{R}^n$ we denote $\|A\|=\max_{i,j} |a_{ij}|$ and $\|x\|= \max _i |x_i|$. The max trace of $A=(a_{ij})\in M_{n\times n}(\mathbb{R}_+)$ equals 
$tr _{\otimes} (A)=\max a_{ii}$.

Let $I_k$ denote the $k\times k$ identity matrix, where $k$ is a positive integer. Also, a matrix $X\in M_{n\times k}(\mathbb{R}_+)$ is called an {\em isometry} in max algebra if $X^t\otimes X=I_k$, and the set of all isometry matrices in max algebra is denoted by $\mathcal{X}_{n\times k}$, where $n,k$ are positive integers with $k\leq n$. For the case $k=n$, $\mathcal{X}_{n\times n}$ is called {\em unitary} in this max algebra and is denoted by $\mathcal{U}_n$. It is known that the set $\mathcal{U}_n$ of unitary matrices equals the group of permutation matrices (see e.g. \cite{Erratum}).

Let $\mathbb{A}=(A_1,\dots ,A_m)$, where $A_i\in M_n(\mathbb{R}_+)$ for all $i=1,\dots ,m$. The max joint numerical range of $\mathbb{A}$ is defined and denoted by (see \cite{TZPA} and the follow-up correction \cite{Erratum})
\begin{equation*}
W_{\max}(\mathbb{A})=\{(x^t\otimes A_1\otimes x,\dots ,x^t\otimes A_m\otimes x):x\in\mathbb{R}_+^n,x^t\otimes x=1\}.
\end{equation*}
The following result on max numerical range $W_{max}(A)$ for one matrix $A\in M_{n}(\mathbb{R})$ was proven in \cite[Theorem 3.7]{TS} and a short proof was also given in \cite[Theorem 2]{TZPA}, \cite{Erratum}.
\begin{theorem}\label{WmaxA}
Let $A=(a_{ij})\in M_{n}(\mathbb{R}_{+})$ be a nonnegative matrix. Then
\[
W_{max}(A)=[a,  b]\subseteq \mathbb{R}_{+},
\] 
where $a=\displaystyle\min_{1\leq i\leq n}a_{ii}$  and  $b=\displaystyle\max_{1\leq i, j\leq n}a_{ij}$.
\end{theorem}
 Let $A, C\in M_n(\mathbb{R}_+)$ and $c=[c_1,\dots ,c_n]^t\in\mathbb{R}_+^n$. The {\em max $c$-numerical range of $A$} is defined and denoted by 
 \begin{equation}
 W_{\max}^c(A)=\{\oplus_{i=1}^nc_i\otimes x_i^t\otimes A\otimes x_i:X=[x_1,\dots ,x_n]\in\mathcal{U}_n\},
 \label{c1}
 \end{equation}
and  the max $C$-numerical range of $A$ is defined and denoted by 
\begin{equation}
W_{\max}^C(A)=\{tr_{\otimes}(C\otimes X^t\otimes A\otimes X):X\in\mathcal{U}_n\}.
\label{bigC1}
\end{equation}
It is clear that $W_{max}^{c}(A)=\{tr_{\otimes}(C\otimes X^{t}\otimes A \otimes X): X\in \mathcal{U}_{n}  \},$ where 
$C=diag(c_{1}, \ldots, c_{n}),   c=[c_{1}, c_{2}, \ldots, c_{n}]^{t}\in \mathbb{R}_{+}^{n}$. So  $W_{\max}^C(A)$  is a generalization of $W_{max}^{c}(A)$ (see \cite{TZPA, Erratum}).

The article is organized in the following manner. In Section 2  we define the rank-$k$ numerical ranges in the max algebra setting, prove some of their basic properties, and explicitly calculate the rank-$k$ numerical range for certain Toeplitz matrices.  In Section 3 we prove some results for max numerical ranges from \cite{TZPA, Erratum}. In Sections 4 and 5 we introduce and study  joint $k$-numerical range and joint $C$-numerical range   in max algebra, respectively. 
\section{\textbf{The rank-$k$ numerical ranges  in max algebra}}
Given an $n\times n$ complex matrix $A,$  that is $A\in M_{n}(\mathbb{C})$, and a positive integer $1\leq k \leq n,$ the conventional rank-$k$ numerical range of $A$ is defined as:
\[
\Lambda_{k}(A)=\{\lambda\in\mathbb{C}: PAP=\lambda P~\textrm{for some rank-$k$ orthogonal projection $P$ on $\mathbb{C}^{n}$}\}.
\]	
From \cite[Proposition 1.1]{CGHK}, we have 
\begin{equation}\label{RankKnumericalRange }
\Lambda_{k}(A)=\{\lambda \in \mathbb{C}: X^{*}AX=\lambda I_{k}, ~\mathrm{for~ some}~ X\in \widetilde{\mathcal{X}}_{n\times k}  \},
\end{equation}
 where $\widetilde{\mathcal{X}}_{n\times k}$ is the set of complex $n\times k$ matrices satisfying 
$X^{*}X= I_{k}$. The sets $\Lambda_{k}(A),$ where $k\in \{1, \ldots, n\},$ are generally referred to as the higher rank numerical ranges of $A$.

It is natural to use a formula analogous to (\ref{RankKnumericalRange }) as a corresponding definition for the  higher rank numerical range of nonnegative matrices in this max algebra setting. Consequently, we make the following definition.
\begin{definition}\label{maxrankknumericalrange}
Let $A \in M_{n}(\mathbb{R}_{+})$ be a nonnegative matrix  and $1\leq k \leq n$ be a positive integer.  The max rank-$k$
numerical range of $A $ is defined and denoted by
\begin{equation*}
\Lambda_{k}^{max}(A)=\{\lambda \in \mathbb{R}_{+}: X^{t}\otimes A\otimes X=\lambda I_{k}~ \mathrm{for~ some}~ X \in \mathcal{X}_{n\times k} \}.
\end{equation*}
\end{definition}
\begin{remark}
Let $A \in M_{n}(\mathbb{R}_{+})$ be a nonnegative matrix  and  let $1\leq k \leq n$ be a positive integer.  Then $\lambda \in \Lambda_{k}^{max}(A)$ if and only if there exists $X=[x_1,\dots ,x_k] \in \mathcal{X}_{n\times k}$  such that
\[
\lambda=x_j^{t}\otimes A \otimes x_j~\textrm{and}~x_r^{t}\otimes A \otimes x_s=0 ~\textrm{for}~ 1\leq r, s \leq k, r\neq s, 1\leq  j \leq k.
\]
\end{remark}
As such, for $k=1,  \Lambda_{k}^{max}(A)$ reduces to the max-numerical range of $A,$  namely,
\begin{equation}\label{eq3}
 \Lambda_{1}^{max}(A)=W_{max}(A)=[a, b],
\end{equation}
where $a =\displaystyle\min_{1\leq i \leq n}a_{ii}$ and $b=\displaystyle\max_{1\leq i, j \leq n}a_{ij}$.

It is readily verified that 
\begin{equation}\label{eq4}
W_{max}(A)= \Lambda_{1}^{max}(A)\supseteq  \Lambda_{2}^{max}(A)\supseteq \cdots \supseteq  \Lambda_{n}^{max}(A).
\end{equation}
The max-numerical range $W_{max}(A)= \Lambda_{1}^{max}(A)$ is a nonempty, compact and convex subset of $\mathbb{R}_{+},$  and $\sigma_{max}(A)\subseteq \sigma_{trop}(A)\subseteq W_{max}(A),$  \cite[Proposition 4]{TZPA}, \cite{Erratum}, where $\sigma_{max}(A)$  and  $\sigma_{trop}(A)$  denote  the set of geometric and algebraic eigenvalues in max algebra, respectively.
First we point out a property of max-higher rank numerical ranges. 
\begin{proposition}
Let $B_1\in M_{n_1}(\mathbb{R}_+),B_2\in M_{n_2}(\mathbb{R}_+)$ and let $A=\left[ \begin{array}{cc}
B_1 & 0\\
0 & B_2
\end{array}
\right] \in M_{n}(\mathbb{R}_+)$, where $n=n_1+n_2$. Moreover, let $1\leq k_1\leq n_1$ and let $1\leq k_2\leq n_2$. Then the following assertion holds:
 \[
 \Lambda_{k_1}^{\max}(B_1)\cap\Lambda_{k_2}^{\max}(B_2)\subseteq\Lambda_k^{\max}(A), 
 \]
 where $k=k_1+k_2$.
\end{proposition}
\begin{proof}
Let $\lambda\in\Lambda_{k_1}^{\max}(B_1)\cap\Lambda_{k_2}^{\max}(B_2)$ be given. Then there are $X_1\in M_{n_1\times k_1}(\mathbb{R}_+),X_2\in M_{n_2\times k_2}(\mathbb{R}_+)$ such that $X_1^t\otimes X_1=I_{k_1}$ and $X_2^t\otimes X_2=I_{k_2}$ and also we have $X_1^t\otimes B_1\otimes X_1=\lambda I_{k_1}$ and $X_2^t\otimes B_2\otimes X_2=\lambda I_{k_2}$. By setting $X=\left[ \begin{array}{cc}
X_1 & 0_{n_1\times k_2}\\
0_{n_2\times k_1} & X_2
\end{array}
\right] \in M_{n\times k}(\mathbb{R}_+)$, we have $X^t\otimes X=I_k$ and  hence
\begin{eqnarray*}
X^t\otimes A\otimes X&=&\left[ \begin{array}{cc}
X_1 & 0\\
0 & X_2
\end{array}
\right]^t\otimes\left[ \begin{array}{cc}
B_1 & 0\\
0 & B_2
\end{array}
\right]\otimes\left[ \begin{array}{cc}
X_1 & 0\\
0 & X_2
\end{array}
\right]\\
&=&\left[ \begin{array}{cc}
X_1^t\otimes B_1\otimes X_1 & 0\\
0 & X_2^t\otimes B_2\otimes X_2
\end{array}
\right]=\lambda I_k.
\end{eqnarray*}
Therefore, $\lambda\in\Lambda_k^{\max}(A),$ which completes the proof.

\end{proof}
The following example shows that the converse inclusion in the above proposition need not hold in general.
\begin{example}\label{ex_higher1}
Let $B_1=\left( \begin{array}{cc}
5 & 0\\
0 & 8
\end{array}\right), B_2=\left( \begin{array}{cc}
10 & 0\\
0 & 12
\end{array}\right)$ and $A=\left( \begin{array}{cc}
B_1 & 0\\
0 & B_2
\end{array}\right)$. Then by Equation (\ref{eq3}) and Definition \ref{maxrankknumericalrange} one has 
\[
 \Lambda_1^{max}(B_1)=[5,8],~ \Lambda_1^{max}(B_2)=[10,12]~  \textit{and} ~ \Lambda_2^{max}(A)=[8,10]
 \]
 and so \[
\Lambda_2^{\max}(A)\neq\Lambda_1^{\max}(B_1)\cap\Lambda_1^{\max}(B_2).
\]

  To prove $\Lambda_2^{max}(A)=[8,10],$ let $\lambda \in \Lambda_2^{max}(A)$ be given. So, there exists 
 $X=\left[\begin{array}{cccc}
x_{1} & x_{2} & x_{3} & x_{4}\\
y_{1} & y_{2} & y_{3} & y_{4}
\end{array}
\right]^{t} \in \mathcal{X}_{4 \times 2}$ such that $X^t\otimes A\otimes X=\lambda I_{2}$. Then we have
\begin{equation*}
\max\{5 x_{1}^{2}, 8x_{2}^{2}, 10x_{3}^{2}, 12x_{4}^{2} \}=\lambda~ \textrm{and}~ \max\{5 y_{1}^{2}, 8y_{2}^{2}, 10y_{3}^{2}, 12y_{4}^{2} \}=\lambda.
\end{equation*}
Now, either $x_{4}$ or $y_{4}$ is zero. Without loss of generality, we may assume $x_{4}=0$.  Then by the first  equality above, we have $\lambda\leq 10$. Similarly, at least one of  $x_{1}$ or $y_{1}$ is zero. Without loss of generality, we may assume $x_{1}=0$.  So if $x_{1}=0,$ then by the first  equality above, we have $\max\{8x_{2}^{2}, 10x_{3}^{2}, 12x_{4}^{2}\}=\lambda$. The minimum value of $\lambda$ occurs when $x_{2}=1,$ so $\lambda\geq 8,$ and hence $\lambda \in [8, 10]$.
Conversely, if $\lambda \in [8, 10],$ by setting
\begin{equation*}
X=\left[\begin{array}{cccc}
1 & 0 & \sqrt{\frac{\lambda}{10}} & 0\\
0 & 1 & 0 & \sqrt{\frac{\lambda}{12}}
\end{array}
\right]^{t} ,
\end{equation*}
then $X^{t}\otimes X=I_{2}$ and $X^{t}\otimes A \otimes X=\lambda I_{k},$ which this implies that $\lambda\in\Lambda_2^{max}(A)$, and the proof is complete.
 

\end{example}

In the following proposition, we state some basic properties of the max rank-$k$ numerical range of nonnegative matrices.

\begin{proposition}\label{pro5}
Let $A\in M_n(\mathbb{R}_+)$. Then the following assertions hold.

\begin{itemize}

\item[(i)] If $A=\alpha I_n,$ where $\alpha\in\mathbb{R}_+,$ then $\Lambda_k^{\max}(A)=\{\alpha\}$ for all $1\leq k\leq n$;

\item[(ii)]   $\Lambda_k^{\max}(\alpha A)=\alpha\Lambda_k^{\max}(A)$  for all $\alpha\in\mathbb{R}_+$ and  $1\leq k\leq n$;

\item[(iii)] $\Lambda_k^{\max}(A^t)=\Lambda_k^{\max}(A)$ for all  $1\leq k\leq n$;

\item[(iv)]   $\Lambda_{n}^{max}(A)\neq\emptyset$ if and only if $A=\lambda I_{n}$.

\item[(v)] If $B\in M_m(\mathbb{R}_+)$ is a principal submatrix of $A$ and $1\leq k\leq m$, then $\Lambda_k^{\max}(B)\subseteq\Lambda_k^{\max}(A)$;

\item[(vi)] If $C=\left[\begin{array}{cc}
A & 0\\
0 & B
\end{array}
\right]$, where $A\in M_{n_1}(\mathbb{R}_+),\ B\in M_{n_2}(\mathbb{R}_+),$ then 
\[
\Lambda_k^{\max}(A)\cup\Lambda_k^{\max}(B)\subseteq\Lambda_k^{\max}(C),  ~ 1\leq k\leq\min\{n_1,n_2\}.
\]
\end{itemize}

\end{proposition}

\begin{proof}
The assertions (i), (ii) and (iii) easily from Definition \ref{maxrankknumericalrange}. Proceeding with the forward implication in the proof of (iv), let $A=(a_{ij})\in M_n(\mathbb{R}_+)$ and $\lambda\in\Lambda_n^{\max}(A)$ be given. Then there exists   $X=[x_1,\dots ,x_n]$ such that $X^t\otimes A\otimes X=\lambda I_n$. Since $X\in\mathcal{U}_n$, there is a $\sigma\in S_n$ such that $x_i=e_{\sigma (i)}$ for all $i=1,\dots ,n$. Hence
\[
x_i^t\otimes A\otimes x_j=e_{\sigma (i)}\otimes A\otimes e_{\sigma (j)}=a_{\sigma (i),\sigma (j)}=\begin{cases}
\lambda & i=j,\\
0 & i\neq j.
\end{cases}
\]
Therefore $A=\lambda I_n$. The converse implication is obvious.

To prove (v), let $B\in M_m(\mathbb{R}_+)$ be formed by considering rows and columns $i_1,\dots ,i_m$ from $A$. Let $\lambda\in\Lambda_k^{\max}(B)$ be given. Then there exists $X=[x_{1}, x_{2}, \ldots, x_{k}]\in \mathcal{X}_{m\times k}$ such that $X^t\otimes B\otimes X=\lambda I_{k}$. Now for any $1\leq i\leq k$, define $y\in\mathbb{R}_+^n$ such that $(y_i)_j=(x_i)_j$ for $j=1,\dots ,m$ and with other  entries $(y_i)_j=0$. Then $Y=[y_1,\dots ,y_k]\in\mathcal{X}_{n\times k}$ and $Y^t\otimes A\otimes Y=\lambda I_k$   and so  $\lambda\in\Lambda_k^{\max}(A)$.

The assertion (vi)  follows from (v),     which completes the proof.
\end{proof}

The following example demonstrates that the reverse inclusion of Proposition \ref{pro5} (v) need not hold in general (also the matrices in Example \ref{ex_higher1} demonstrate this fact).
\begin{example}
Let $A=\begin{bmatrix}
20 & 9\\
7 &1
\end{bmatrix}, B=\begin{bmatrix}
2 &6\\
5 &3
\end{bmatrix}$ and $C=\begin{bmatrix}
A & 0\\
0 & B
\end{bmatrix}$. Then one has 
\[
\Lambda_1^{\max}(A)=[1,20],\ \Lambda_1^{\max}(B)=[2,6],  \Lambda_1^{\max}(C)=[1,20]
\]
and 
\[
\Lambda_2^{\max}(A)=\Lambda_2^{\max}(B)=\emptyset,  \Lambda_2^{\max}(C)=[2,6],
\]
hence
\[
\Lambda_1^{\max}(C)=\Lambda_1^{\max}(A)\cup\Lambda_1^{\max}(B),
\]
but
\[
\Lambda_2^{\max}(C)\neq\Lambda_2^{\max}(A)\cup\Lambda_2^{\max}(B).
\]
Also, 
\[
\Lambda_3^{\max}(C)=\Lambda_4^{\max}(C)=\emptyset.
\]

To prove that $\Lambda_2^{\max}(C)=[2,6],$ suppose that $\lambda \in [2, 6]$ is given.

If $2\leq \lambda\leq \frac{81}{20},$ by setting 
\begin{equation*}
X=\left[\begin{array}{cccc}
0 & 0 & 1 & \frac{\lambda}{6}\\
\frac{\lambda}{9} & 1 & 0 & 0
\end{array}
\right]^{t} ,
\end{equation*}
then $X^{t}\otimes X=I_{2}$ and $X^{t}\otimes C \otimes X=\lambda I_{2}$.

If  $\frac{81}{20}\leq \lambda\leq 6,$  by setting 
\begin{equation*}
X=\left[\begin{array}{cccc}
0 & 0 & 1 & \frac{\lambda}{6}\\
\sqrt{\frac{\lambda}{20}} & 1 & 0 & 0
\end{array}
\right]^{t} ,
\end{equation*}
then $X^{t}\otimes X=I_{2}$ and $X^{t}\otimes C \otimes X=\lambda I_{2}$. Thus $[2,6] \subseteq \Lambda_2^{\max}(C)$. 

Conversely, if $\lambda \in \Lambda_2^{\max}(C),$ then there exists some 
\begin{equation*}
X=\left[\begin{array}{cccc}
x_{1} & x_{2} & x_{3} & x_{4}\\
y_{1} & y_{2} & y_{3} & y_{4}
\end{array}
\right]^{t} \in \mathcal{X}_{4\times 2}
\end{equation*}
such that $X^{t}\otimes C\otimes X=\lambda I_{2}$. Then
\begin{equation}
\max\{20 x_{1}^{2}, 9 x_{1}x_{2}, x_{2}^{2}, 2x_{3}^{2}, 6x_{3}x_{4}, 3x_{4}^{2}\}=\lambda,  
\end{equation}
\begin{equation}
\max\{20 y_{1}^{2}, 9 y_{1}y_{2}, y_{2}^{2}, 2y_{3}^{2}, 6y_{3}y_{4}, 3y_{4}^{2}\}=\lambda.
\end{equation}
For the case $\lambda<2,$  from Relations (2.4) and (2.5), we deduce that $x_i,y_i\neq 1$ for $i=1,3,4$. So,  $x_{2}=1$ and $y_{2}=1$, which is a contradiction. Similarly, for the case $\lambda>6,$ we deduce that $20x_{1}^{2}=\lambda$ or $9x_{1}x_{2}=\lambda,$ which implies  $x_{1}\neq 0$. In the same way, we find  $y_{1}\neq 0,$ which also leads to a contradiction. Therefore, we conclude that  $\lambda \in [2, 6],$ completing the proof.


To prove that $\Lambda_3^{\max}(C)=\emptyset,$ we assume that  $\Lambda_3^{\max}(C)\neq \emptyset$. Then there exists some $\lambda\in \Lambda_3^{\max}(C)$ such that $X^{t}\otimes C\otimes  X=\lambda I_{3},$ where 
\begin{equation*}
X=\left[\begin{array}{ccc}
x_{1} & y_{1} & z_{1}\\
x_{2} & y_{2} & z_{2}\\
x_{3} & y_{3} & z_{3}\\
x_{4} & y_{4} & z_{4}
\end{array}
\right]\in \mathcal{X}_{4\times 3}.
\end{equation*}
Then
\begin{equation}\label{2.6}
\max\{20x_{1}y_{1}, 9x_{1}y_{2}, 7x_{2}y_{1}, x_{2}y_{2}, 2x_{3}y_{3}, 6x_{3}y_{4}, 5x_{4}y_{3}, 3x_{4}y_{4} \}=0,
\end{equation}
\begin{equation}\label{2.7}
\max\{20x_{1}z_{1}, 9x_{1}z_{2}, 7x_{2}z_{1}, x_{2}z_{2}, 2x_{3}z_{3}, 6x_{3}z_{4}, 5x_{4}z_{3}, 3x_{4}z_{4} \}=0,
\end{equation}
and
\begin{equation}\label{2.8}
\max\{20y_{1}z_{1}, 9y_{1}z_{2}, 7y_{2}z_{1}, y_{2}z_{2}, 2y_{3}z_{3}, 6y_{3}z_{4}, 5y_{4}z_{3}, 3y_{4}z_{4} \}=0.
\end{equation}
Since the first column of  $X$ contains at least one nonzero entry first assume that $x_{1}\neq 0$. From Relation (\ref{2.6}), we deduce $y_{1}=y_{2}=0$. Similarly, from Relation (\ref{2.7}), we conclude that  $z_{1}=z_{2}=0$. On the other hand, the second column must also contain at least one nonzero entry. If $y_{3}\neq 0,$ then from Relation (\ref{2.8}), we again have $z_{3}=z_{4}=0,$ which implies that the last column is zero, contradicting the definition of $X$. For the case $y_{4}\neq 0,$  similarly (\ref{2.8}) leads to $z_{3}=z_{4}=0$, which also implies that the last column is zero, again contradicting the definition of $X$. The remaining cases of nonzero entries in the first column of $X$ follow in a similar manner. Therefore, $\Lambda_3^{\max}(C)=\emptyset$ and the proof is complete.
 
 \end{example}

We now present a result on $\Lambda_{k}^{max}(A)$ for particular Toeplitz matrices $A$.

\begin{theorem}
Suppose that $n \ge 3$ and that $A \in M_{n}(\mathbb{R}_{+})$ is the following  Toeplitz  matrix 
\[
A=\left [ \begin{array}{cccccc}
a_{0}  &  a_{-1} & a_{-2}  & \cdots & a_{-(n-2)} & 0 \\ 
a_{1}  &   a_{0}  & a_{-1} & \ddots  &  & a_{-(n-2)}  \\
a_{2} &    a_{1}  & \ddots &  \ddots & \ddots &  \vdots\\
\vdots & \ddots & \ddots & \ddots &  a_{-1} &  a_{-2}\\
 a_{n-2} &  & \ddots & a_{1} &  a_{0} & a_{-1} \\
 0  &  a_{n-2} & \cdots & a_{2} & a_{1} & a_{0}
\end{array}
\right],
\]
where $a_{i}\neq 0,~ -(n-2) \leq i \leq (n-2)$.
Then
\[
\Lambda_{1}^{max}(A)=[a_{0},  \displaystyle \max_{-(n-2) \leq i \leq (n-2)}a_{i}], ~ \Lambda_{2}^{max}(A)=\{ a_{0} \}
\]
and
\[
\Lambda_{k}^{max}(A)=\emptyset,~ \forall~   3 \leq k \leq n.
\]
\end{theorem}
\begin{proof}
By (\ref{eq3}) we have
\[
\Lambda_{1}^{max}(A)=[a_{0},  \displaystyle \max_{-(n-2) \leq i \leq (n-2)}a_{i}].
\]
Now, we show that $ \Lambda_{2}^{max}(A)=\{ a_{0} \}$. At first, we show that $\{ a_{0} \}\subseteq\Lambda_{2}^{max}(A)$. For this, put $X=\begin{bmatrix}
1 &0\\
0&0\\
\vdots &\vdots\\
0&1
\end{bmatrix}\in\mathcal{X}_{n\times 2}$, then we have $X^t\otimes A\otimes X=a_0I_2$. Therefore, $a_0\in\Lambda_{2}^{max}(A)$. Now, let $\lambda\in\Lambda_{2}^{max}(A)$ be given. So there exists $X=(x_{ij})\in\mathcal{X}_{n\times 2}$ such that $X^t\otimes A\otimes X=\lambda I_2$. Therefore, we have $(X^t\otimes A\otimes X)_{12}=0$, and so
\[
\displaystyle\max_{1\leq i,j\leq n}\{x_{1i}a_{ij}x_{j2}\}=0.
\]
Hence
\begin{eqnarray*}
&&x_{11}x_{j2}=0\hspace*{3cm}\forall j=1,\dots ,n-1\\
&&x_{21}x_{j2}=0\hspace*{3cm}\forall j=1,\dots ,n\\
&&\hspace*{3.4cm} \vdots\\
&&x_{(n-1)1}x_{j2}=0\hspace*{2.3cm}\forall j=1,\dots ,n\\
&&x_{n1}x_{j2}=0\hspace*{3cm}\forall j=2,\dots ,n.\\
\end{eqnarray*}
It readily follows that  $X=\begin{bmatrix}
1 &0\\
0&0\\
\vdots &\vdots\\
0 &1
\end{bmatrix}$ or $X=\begin{bmatrix}
0 &1\\
0&0\\
\vdots &\vdots\\
1 &0
\end{bmatrix}$. Then $X^t\otimes A\otimes X=a_0I_2$. Therefore $\lambda =a_0$. 

Next we verify that there is no  $\lambda\in\mathbb{R}_+,  X\in\mathcal{X}_{n\times 3}(\mathbb{R}_+)$ such that $X^t\otimes A\otimes X=\lambda I_3$. Assume that such $\lambda\in\mathbb{R}_+$  and $X\in \mathcal{X}_{n\times 3}$ exist for $n>3,$ so that $X^{t}\otimes A\otimes X=\lambda I_{3}$. Consider the $(1, 2)-$entry of the matrix product $(X^{t}\otimes A\otimes X)_{12}=0$. This implies that 
\begin{equation}\label{xtAx12}
\max_{1\leq i, j\leq n}\{x_{i1}a_{ij}x_{j2}  \}=0.
\end{equation}
We will prove  that $x_{i1}=0$  for all $1\leq i\leq n$, which will be in contradiction with $X\in\mathcal{X}_{n\times 3}(\mathbb{R}_+)$.

Assume $x_{11}\neq 0$. Note that all entries in the first row of $A$ are nonzero except for the last entry, i.e., $A_{1j}\neq 0$ for $1\leq j\leq n-1$. From the equation above, it follows that $x_{j2}=0$ for   $1\leq j \leq n-1,$ implying that the second column of $X $ has $x_{n2}=1$.

Now, consider the $(3, 2)-$entry of the matrix product  $(X^{t}\otimes A\otimes X)_{32}=0$. This implies that 
\begin{equation}\label{XtAX32}
\max_{1\leq i, j\leq n}\{x_{i3}a_{ij}x_{j2}  \}=0.
\end{equation}
We already know that $x_{n2}=1,$ and since $A_{in} \neq 0$ for $2\leq i \leq n,$ it follows that  $x_{i3}=0$ for $2\leq i\leq n$.
 This means that all entries in the third column of $X$ are zero except for the first entry, implying $x_{13}=1$. However, this leads to a contradiction, as $x_{11}\neq 0,$ which is not possible given the structure of $X$.
To extend this argument, consider the entry $x_{21}$. If  $x_{21}\neq 0,$ then since $A_{2j}\neq 0$ for all $1\leq j\leq n,$  we have $x_{j2}=0$ for all $1\leq j\leq n,$ contradicting the nonzero requirement for the second column of $X$. By similar reasoning, we can show that $x_{31}=\cdots=x_{(n-1)1}=0$.  Finally, suppose $x_{n1}\neq 0$. Since $A_{nj}\neq 0$ for $2\leq j \leq n$ we have $x_{j2}=0$ for $2\leq j\leq n$. This implies  that all the entries in the second column of $X$ are zero except for the first entry, leading to  $x_{12}=1$. But we know  $A_{i1}\neq 0$ for $1\leq i\leq n-1,$ which implies  $x_{i3}=0$ for $1\leq i\leq n-1$. This leads to $x_{31}=1,$ which, together with $x_{n1}\neq 0,$ again contradicts the definition of  $X$.

For the case $n=3,$ let
\begin{equation*}
A=\left [ \begin{array}{ccc}
a_{0}  &  a_{-1} & 0 \\ 
a_{1}  &   a_{0}  & a_{-1}\\
0  &    a_{1}  &  a_{0}
\end{array}
\right].
\end{equation*}
 We want to show that $\Lambda_{3}^{max}(A)=\emptyset$ as well. Suppose, by contradiction, that $\Lambda_{3}^{max}(A)\neq \emptyset$.  Then there exists some $\lambda\in \Lambda_{3}^{max}(A)$ and $X\in \mathcal{X}_{3\times 3}$ such that $X^{t}\otimes A \otimes X=\lambda I_{3}$. Since $X$ is a permutation matrix containing only zeros and ones, $X^{t}\otimes A \otimes X$ is a permutation of $A,$ 
 which cannot be equal to $\lambda I,$ completing the proof.
\end{proof}
\begin{example}
 Let 
\begin{equation*}
A=\left [ \begin{array}{cccc}
3  &  4 & 2  & 0 \\ 
5  &   3  & 4 & 2\\
2 &    5  & 3 &  4\\
0  &  2  &  5  &   3  
\end{array}
\right],
B=\left [ \begin{array}{ccccc}
3  &  4 & 2  & 7 & 0 \\ 
8  &   3  & 4 & 2 & 7\\
6 &    8  & 3 &  4 & 2\\
2  &  6  &  8  &   3  & 4\\
0 & 2 & 6 & 8 & 3
\end{array}
\right]
\text{and}\ 
C=\left [ \begin{array}{cccccc}
4  &  7 & 6  & 3 & 5 & 0\\ 
2  &   4  & 7 & 6 & 3 & 5\\
8 &    2  & 4 & 7 & 6 & 3\\
9  &  8  &  2  &   4  & 7 & 6\\
1 & 9 & 8 & 2 &  4 & 7 \\
0 & 1 & 9 & 8 &  2 & 4 
\end{array}
\right].
\end{equation*}
Then
\begin{eqnarray*}
\Lambda_{1}^{max}(A)=[3, 5],\Lambda_{2}^{max}(A)=\{3\}~ \mathrm{and}\  \Lambda_{3}^{max}(A)=\Lambda_{4}^{max}(A)=\emptyset, \\
\Lambda_{1}^{max}(B)=[3, 8],~\Lambda_{2}^{max}(B)=\{3\} ~ \mathrm{and}~ \Lambda_{k}^{max}(B)=\emptyset~ \forall~3 \leq k \leq 5,\\
\Lambda_{1}^{max}(C)=[4, 9],~\Lambda_{2}^{max}(C)=\{4\} ~ \mathrm{and}  ~\Lambda_{k}^{max}(C)=\emptyset~ ~\forall~3 \leq k \leq 6.
\end{eqnarray*}
\end{example}


\medskip

Higher rank numerical radius of $A$ is defined by (see e.g. \cite{CGLTW})
\begin{equation*}
\omega_{k}^{\Lambda}(A)=\sup\{\vert \mu \vert: \mu \in \Lambda_{k}(A) \},
\end{equation*}
where we use  the convention that $\omega_{k}^{\Lambda}(A)=-\infty$ if $\Lambda_{k}(A)=\emptyset$. Next we consider an analogue of higher rank numerical radius in max algebra setting.
\begin{definition}
Let $A\in M_n(\mathbb{R}_+)$ and let $1\leq k \leq n$ be a positive integer. The max higher rank numerical radius of $A$ is defined and denoted by
\begin{equation*}
\omega_{k}^{\Lambda_{max}}(A)=\sup\{\lambda: \lambda\in \Lambda_{k}^{max}(A) \},
\end{equation*}
where we use the convention that $\omega_{k}^{\Lambda_{max}}(A)=-\infty$ if $\Lambda_{k}^{max}(A)=\emptyset$.
\end{definition}

A basic inequality concerning the higher rank numerical radius is provided in the following.

\begin{proposition}
Let $Z\in M_{n}(\mathbb{R}_{+})$ and let $1\leq k \leq n$ be a positive integer. If $Z=\displaystyle\bigoplus_{i=1}^{s}x_{i}\otimes y_{i}^{t}$ for some $x_{i}, y_{i}\in \mathbb{R}_{+}^{n},  1\leq i \leq s$. Then
\[
\Lambda_{k}^{max}(Z)\subseteq \displaystyle\bigcup_{j=1}^{s} W_{max}(x_j\otimes y_j^{t}).
\]
Consequently,
\begin{equation*}
\omega_{k}^{\Lambda_{max}}(A)\leq \displaystyle\bigoplus_{j=1}^{s}\Vert x_j\otimes y_j^{t}\Vert.
\end{equation*}
\end{proposition}
\begin{proof}
Let $\lambda \in \Lambda_{k}^{max}(Z)$ be given. By Definition \ref{maxrankknumericalrange} 
\[
X^{t}\otimes Z\otimes X=\lambda I_{k}~ \mathrm{for~ some}~ X=[X_{1}, \ldots, X_{k}]\in \mathcal{X}_{n\times k}.
\]
Hence  there exist $i\leq k\leq k$ and $1\leq j \leq s$ such that
\[
\lambda=X_{i}^{t}\otimes x_{j}\otimes y_{j}^{t}\otimes X_{i}, 1\leq i \leq k.
\]
So $\lambda\in W_{\max}(x_j\otimes j_j^t)$, because $X_i\in\mathbb{R}_+^n$ and $X_i^t\otimes X_i=1$. Then 
\[
\lambda\in \displaystyle\bigcup_{j=1}^{s} W_{max}(x_j\otimes y_j^{t}),
\]
which completes the proof. 
\end{proof}


\section{Some properties of max numerical ranges}
Recall that the conventional numerical radius of $A\in M_{n}(\mathbb{C})$ is defined by
\begin{equation*}
\omega(A)=\sup\{\vert z \vert: z\in W(A)\}.  
\end{equation*}

The max numerical radius of $A$ equals
\[
\Vert A \Vert=\sup\{z: z\in W_{max}(A)\}=\displaystyle\max_{1\leq i, j\leq n}a_{ij},
\]
which  can be considered as operator norm in max algebra sense, since
$$\|A\|= \max \{\|A\otimes x\|: \; \|x\|\le 1, \; x\ge 0\}.$$

 In the following remark we state some of its basic properties, which are well-known or easy to verify. 
 \begin{remark}\label{th4}
Let  $A\in M_n(\mathbb{R}_+)$. Then the following assertions are all valid:
 \begin{itemize}
 \item[(i)] $\Vert A \Vert\geq 0$;
 \item[(ii)] $\Vert A \Vert=0$ if and only if $A=0$;
 \item[(iii)] $\Vert \alpha A\Vert=\alpha \Vert A \Vert$, where $\alpha\in\mathbb{R}_+$; 
 \item[(iv)] $\Vert A \Vert=\Vert A^t \Vert$;
 \item[(v)] $\Vert B \Vert=\Vert A \Vert$, where $B=U^t\otimes A\otimes U$ such that  $U\in\mathcal{U}_n$;
 \item[(vi)] $\Vert A\oplus B \Vert=\Vert A \Vert \oplus \Vert B \Vert$, where  $B\in M_n(\mathbb{R}_+)$;
 \item[(vii)] $\Vert A\otimes B \Vert\leq \Vert A\Vert \Vert B\Vert$;
 \item[(viii)] $\Vert A_{\otimes}^m \Vert \leq \Vert A \Vert^{m},$  where $m$ is a positive integer.
\item[(ix)]  If $x,y\in\mathbb{R}_+^n$ are such that $\|x\| \le 1$, $\|y\| \le 1$, then  
$\Vert A\otimes x\otimes y^t\Vert\leq\Vert A\Vert.$
 \end{itemize}
 \end{remark}

Let $A \in M_{n}(\mathbb{R}_{+})$ and   $1 \le k \leq n$.   The {\em max $k$-numerical range $W_{\max}^k(A)$ of $A$} in  max algebra was introduced in \cite[Section 4]{TZPA} (see also \cite{Erratum}) and is defined by
%
\begin{eqnarray*}
W_{max}^{k}(A) &=&\{\bigoplus_{i=1}^{k} \displaystyle (x^{(i)})^{t}\otimes A \otimes x^{(i)}:  ~X=[x^{(1)}, x^{(2)}, \ldots, x^{(k)}] \in \mathcal{X}_{n \times k} \}\\
&=& \{tr_{\otimes}(X^{t}\otimes A \otimes X):~ X=[x^{(1)}, x^{(2)}, \ldots, x^{(k)}] \in \mathcal{X}_{n \times k}   \}. 
\end{eqnarray*}
Note that $W_{\max}(A)=W_{\max}^1(A)\supseteq W_{\max}^2(A)\supseteq\cdots \supseteq W_{\max}^n(A)$ (\cite[Theorem 3.4(v)]{Erratum}).

The following result was proved in \cite[Theorem 3.13]{Erratum}. 

 \begin{theorem}\label{th1}
 Suppose that $A=(a_{ij})\in M_n(\mathbb{R}_+)$ and let $1\leq k< n$ be a positive integer. Then
 \[
 W_{\max}^k(A)=[c,d],
 \]
 where $c=\min\{\displaystyle\oplus_{j=1}^ka_{i_ji_j}:~1\leq i_1<i_2<\cdots <i_k\leq n\}$ and $d=\displaystyle\max_{1\leq i,j\leq n}a_{ij}$. 
 
  Moreover,
    \begin{equation}\label{WmaxNnumberone}
W_{max}^{n}(A)=\{\displaystyle\max_{1\leq i\leq n}a_{ii} \}.
\end{equation}
 \end{theorem}
 

Next we list some observations on max $k$-numerical ranges.

\begin{remark}\label{remarkWkAZ}
(i) Let $A,B\in M_n(\mathbb{R}_+)$ and let $1\leq k \leq n$ be a positive integer. If  $W_{\max}^{k}(A\otimes x\otimes y^t)=W_{\max}^{k}(x\otimes y^t\otimes B)$ for all $x,y\in\mathbb{R}_+^n,$ by using \cite[Theorem 3.4(i)]{Erratum}(\cite[Theorem 4 (i)]{TZPA}) we have
\[
W_{\max}^{k}(A\otimes Z)\subseteq \displaystyle \bigoplus_{i=1}^{r}W_{\max}^{k}(x_{i}\otimes y_{i}^t\otimes B),
\]
where $Z=\displaystyle \bigoplus_{i=1}^{r}x_{i}\otimes y_{i}^{t}$.

(ii) Also,  if  $\Vert A\otimes x\otimes y^t\Vert=\Vert x\otimes y^t\otimes B\Vert$  $\mathrm{ for~ all} ~x,y\in\mathbb{R}_+^n,$ then
\[
\Vert A\otimes Z \Vert = \displaystyle \bigoplus_{i=1}^{r}\Vert x_{i}\otimes y_{i}^{t}\otimes B \Vert=\|Z\otimes B\|,
\]
where $Z=\displaystyle \bigoplus_{i=1}^{r}x_{i}\otimes y_{i}^{t}$.
\end{remark}

%
%
%
\begin{remark}
Let $A=x\otimes y^{t}\in M_{n}(\mathbb{R}_{+}),$ where  $x=[x_{1},\ldots , x_{n}]^t,  y=[y_{1}, \ldots , y_{n}]^t\in \mathbb{R}_{+}^{n}$ are two nonnegative vectors and  let $1\leq k \leq n$ be a positive integer. 
 Then it follows from Theorem \ref{th1} that
\begin{equation*}\label{Wkmaxxy}
W_{max}^{k}(A)=[\displaystyle\min\{\displaystyle\oplus_{j=1}^{k}x_{i_{j}}y_{i_{j}}: 1\leq i_{1}<i_{2}<\cdots<i_{k}< n \},  \displaystyle\max_{1\leq i,  j \leq n}(x_{i}y_{j})]
\end{equation*}
and 
$$W_{max}^{k}(A)= \{\max_{i=1, \ldots , n} x_i y_i\}$$
 since $a_{ij}=x_{i}y_{j}$  $\mathrm{for~ all}$~ $i, j=1, \ldots, n$.  Hence 
\[
W_{max}(A)=\displaystyle[\min_{1\leq i \leq n} (x_{i}y_{i}),  \max_{1\leq i,  j \leq n} (x_{i}y_{j})] \; \; \;\mathrm{and} \; \; \; \Vert A\Vert=\displaystyle\max_{1\leq i,  j \leq n} (x_{i}y_{j}).
\]
\end{remark}
\begin{remark}
Let $A=(a_{ij})\in M_n(\mathbb{R}_+)$ and $1\leq k < n$ be a positive integer. If
\[
\min\{\displaystyle\oplus_{j=1}^{k}a_{i_{j}i_{j}}: 1\leq i_{1}<i_{2}<\cdots<i_{k}\leq n \}=\displaystyle\min_{1\leq i\leq n}a_{ii},
\]
then $W_{max}^{k}(A)=W_{max}(A)$. Consequently, if there exist at least $k$ distinct indices $ i_{1}, \ldots, i_{k}$  such that  
$
a_{i_{1}i_{1}}=\cdots=a_{i_{k}i_{k}}=\displaystyle\min_{1\leq i \leq n}a_{ii},
$
then  $W_{max}^{k}(A)=W_{max}(A)$.
\end{remark}

\begin{remark}\label{Wmaxconv}
 Let $A=\begin{bmatrix}
 D & 0_{n_1\times n_2}\\
 0_{n_2\times n_1} & C
 \end{bmatrix}$, where $D\in M_{n_1}(\mathbb{R}_+)$ and $C\in M_{n_2}(\mathbb{R}_+)$. Then $W_{\max}(A)=conv_{\otimes}\displaystyle(W_{\max}(D)\cup W_{\max}(C))$.
 \end{remark}
 The following example shows that an analogue of Remark \ref{Wmaxconv} does not hold for $k>1$ in general.
 \begin{example}
 Let $A=\begin{bmatrix}
 D & 0\\
 0 & C
 \end{bmatrix}\in M_{6}(\mathbb{R}_{+}),$  where 
 \[
 D=\begin{bmatrix}
 3 & 2 & 4\\
 5 & 7 & 8\\
 2 & 3 & 6
 \end{bmatrix},~     
 C=\begin{bmatrix}
 5 & 3 & 10\\
 2 & 4 & 9\\
 3 & 8 & 7
 \end{bmatrix}.
 \]
 Thus
 \[
 W_{\max}^{2}(D)=[6,  8], ~W_{\max}^{2}(C)=[5,  10]
 \]
 and
 \[
conv_{\otimes}\displaystyle(W_{\max}^{2}(D)\cup W_{\max}^{2}(C))=[5, 10],~ W_{\max}^{2}(A)=[4, 10].
 \]
 \end{example}
 \begin{proposition}
Let  $A=(a_{ij})\in M_n(\mathbb{R}_+)$ and let $\frac{n}{2}\leq k<n$. Then  
\[
W_{\max}^k(A)=\displaystyle\min_{1\leq i_{1}<\cdots<i_{k}\leq n}\oplus_{j=1}^ka_{i_ji_j}\oplus W_{\max}^{n-k}(A).
\]
\end{proposition}
\begin{proof}
Since $n-k\leq k,$  it follows
\[
\displaystyle\min_{1\leq i_{1}<\cdots<i_{n-k}\leq n}\oplus_{j=1}^{n-k}a_{i_ji_j}\leq \displaystyle\min_{1\leq i_{1}<\cdots<i_{k}\leq n}\oplus_{j=1}^k a_{i_ji_j}
\]
and hence the result follows from Theorem  \ref{th1}.
\end{proof}

\begin{proposition}\label{pro2}
Let $1\le k <n$ and $A=\begin{bmatrix}
A_1& & 0\\
&\ddots\\
0& &A_k
\end{bmatrix}\in M_{n}(\mathbb{R}_{+})$ be a nonnegative matrix, where  $A_i\in M_{n_i}(\mathbb{R}_+),  i=1,\dots ,k$ and let  $W_{\max}(A_1)=\dots =W_{\max}(A_k)$. Then 
\[
W_{\max}(A)=W_{\max}^k (A).
\]
Consequently, $W_{\max}(A)=W_{\max}^m (A)$  for all $1\leq m \leq k$.
\end{proposition}
\begin{proof}
Let $A=(a_{ij})\in M_n(\mathbb{R}_+)$ and let $A_l=(a_{ij}^{(l)})\in M_{n_i}(\mathbb{R}_+), 1\leq l\leq k$. Since $W_{\max}(A_1)=\dots =W_{\max}(A_k)$, so 
\[
\min_{1\leq i\leq n_1}a_{ii}^{(1)}=\dots =\min_{1\leq i\leq n_k}a_{ii}^{(k)}:=a
\]
and 
\[
\max_{1\leq i,j\leq n_1}a_{ij}^{(1)}=\dots =\max_{1\leq i,j\leq n_k}a_{ij}^{(k)}:=b.
\] 
These equalities  imply that
\[
\min\{\oplus_{j=1}^ka_{i_ji_j}:1\leq i_1<\dots<i_k\leq n\}=a
\]
and 
\[
\max_{1\leq i,j\leq n}a_{ij}=b.
\]
Hence by Theorem \ref{th1}, $W_{\max}^k(A)=[a,b]=W_{\max}(A)$.
\end{proof}
\begin{remark}
In Proposition \ref{pro2} we may replace the condition  $W_{\max}(A_{1})=\cdots =W_{\max}(A_{k})$ with a weaker condition
\[
\min_{1\leq i\leq n_1}a_{ii}^{(1)}=\dots =\min_{1\leq i\leq n_k}a_{ii}^{(k)}=a.
\]
\end{remark}

The following results are  max algebra analogues of \cite[Theorem 2.7, Corollary 2.8-2.10]{CGLTW}.
\begin{proposition}\label{pro3}
Let $A\in M_n(\mathbb{R}_+)$ such that $A=\alpha U\oplus C$, where $\alpha\in\mathbb{R}_+, U\in\mathcal{U}_n$ and $C=(c_{ij})$ with $c_{ij}\leq\alpha$ for all $1\leq i,j\leq n$. Then $\Vert A\otimes B\Vert=\Vert B\otimes A\Vert$ for all $B\in M_n(\mathbb{R}_+)$. 
\end{proposition}

\begin{proof}
Let $A=\alpha U\oplus C,$ where $\alpha ,C$ satisfy the desired assumptions and  let $B=(b_{ij})\in M_n(\mathbb{R}_+)$ be given. Then it follows that 
\[
\Vert A\otimes B\Vert=\alpha \max_{1\leq i,j\leq n}b_{ij}=\Vert B\otimes A\Vert.
\]

\end{proof}

\begin{remark}
Considering the following example by S. Gaubert (\cite{Gaubert})
\[A=\begin{bmatrix}
1&1&1\\
1&0&0\\
1&0&0
\end{bmatrix},
\]
 an elementary computation implies that
 \[
\Vert A\otimes B\Vert = \Vert B\otimes A\Vert = \displaystyle\max_{1\leq i,j\leq 3}
b_{ij}
 \]
for all $3 \times 3$ nonnegative matrices $B$, but one cannot write $A = \alpha U \oplus C$, where $U\in\mathcal{U}_n$, $\alpha \ge 0$ and $\|C\|\le \alpha$,  because for all permutations $\sigma\in  S_3$,  $\displaystyle\prod_{i=1} ^3 a_{i\sigma (i)}=0$. This shows that the converse of the above proposition is not correct in general.
\end{remark}

\begin{remark}
 Let $x, y\in \mathbb{R}_{+}^{n}$ be such that $x^T \otimes y=0$ and $\|x\|=\|y\|=1$. If $A=x\otimes y^{t},  B=x\otimes x^{t},$  one can easily see
\[
A\otimes B=x\otimes y^{t}\otimes x \otimes x^{t}=O\in M_{n}(\mathbb{R}_{+})
\]
and
\[
B\otimes A=x\otimes x^{t}\otimes x \otimes y^{t}=x\otimes y^{t}=A.
\]
Also we have
\[
W_{max}(A\otimes B)=\{0\},~~~W_{max}(B\otimes A)=W_{max}(A)=[0,1]
\]
and 
 $\Vert A\otimes B\Vert=0,\ \   \Vert B\otimes A\Vert=1$.
\end{remark}

%
%
%
%
%

\begin{remark}
 Let $A=(a_{ij}),B\in M_n(\mathbb{R}_+)$ such that $A\otimes B=I=B\otimes A$. By  \cite[Theorem 1.1.3]{B} $A=P \otimes D$ for $P \in \mathcal{U}_n$ and a positive diagonal matrix $D$ and so
\[
\Vert A\Vert=\|D\| \;\;\mathrm{and}  \;\;\ \ \Vert B\Vert=\frac{1}{\displaystyle\min_{1\leq i,j\leq n, a_{ij}\neq 0} a_{ij}}.
\]
\end{remark}

\begin{proposition}
 Let  
\[
A=\begin{bmatrix}
I_{r} & O\\
O & A_{1}
\end{bmatrix},~~~B=\begin{bmatrix}
C & O \\
O &  B_{1} 
\end{bmatrix} \in M_{n}(\mathbb{R}_{+}),
\]
where  $A_{1},  B_{1}\in M_{n-r}(\mathbb{R}_{+})$ and $C\in M_{r}(\mathbb{R}_{+})$. If $c_{ii}=0$ for some    $i=1, \ldots, r$   and $\displaystyle\Vert A_{1}\otimes B_{1}\Vert=\displaystyle\Vert B_{1}\otimes A_{1}\Vert,$ then
\[
W_{max}(A\otimes B)=W_{max}(B\otimes A)=[0, \displaystyle\| (A_{1}\otimes B_{1})\|\oplus \|C\|].
\]
\end{proposition}
\begin{proof}
By computing $A\otimes B$ and $B\otimes A,$  one has
\[
A\otimes B=\begin{bmatrix}
C & O\\
O & A_{1}\otimes B_{1}
\end{bmatrix},~ B\otimes A=\begin{bmatrix}
C & O\\
O & B_{1}\otimes A_{1}
\end{bmatrix}.
\]
Since $c_{ii}=0$ for some    $i=1, \ldots, r$  and $\displaystyle\Vert A_{1}\otimes B_{1}\Vert=\displaystyle\Vert B_{1}\otimes A_{1}\Vert$ the result follows from Theorem  \ref{WmaxA}. 
\end{proof}
The following proposition is proved similarly by Theorem \ref{th1}.
\begin{proposition}
 Let  
\[
A=\begin{bmatrix}
I_{r} & O\\
O & A_{1}
\end{bmatrix},~~~B=\begin{bmatrix}
C & O \\
O &  B_{1} 
\end{bmatrix} \in M_{n}(\mathbb{R}_{+}),
\]
where  $A_{1},  B_{1}\in M_{n-r}(\mathbb{R}_{+})$ and $C\in M_{r}(\mathbb{R}_{+})$. If $c_{ii}=0$ for some $i=i_{1}, i_{2}, \ldots, i_{s},~ 1\leq s \leq r$,     then
\[
W_{max}^{k}(A\otimes B)=\begin{cases}
 [0,  \Vert (A_{1}\otimes B_{1})\Vert \oplus \Vert C\Vert] &  1\leq k\leq s \\ 
  [c, \Vert (A_{1}\otimes B_{1})\Vert \oplus \Vert C\Vert]  & s+1\leq k  < n
\end{cases}, 
\]
where  $c= \displaystyle\min_{1\leq i_1<\dots <i_k\leq n} \displaystyle\bigoplus_{j=1}^{k}( (A \otimes B)_{i_{j}i_{j}}$, and similarly
\[
W_{max}^{k}(B\otimes A)=\begin{cases}
 [0,  \Vert (B_{1}\otimes A_{1})\Vert \oplus \Vert C\Vert] &  1\leq k\leq s \\ 
  [d, \Vert (B_{1}\otimes A_{1})\Vert \oplus  \Vert C\Vert]  & s+1\leq k < n
\end{cases}, 
\]
where  $d= \displaystyle\min_{1\leq i_1<\dots <i_k\leq n} \displaystyle\bigoplus_{j=1}^{k}( B \otimes A)_{i_{j}i_{j}}$. 

 If   $\displaystyle\Vert A_{1}\otimes B_{1}\Vert=\displaystyle\Vert B_{1}\otimes A_{1}\Vert,$ then
\[
W_{max}^{k}(A\otimes B)=W_{max}^{k}(B\otimes A),     1\leq k \leq s, \; 1 \le k <n.
\]
\end{proposition}
\begin{example}
 Let   
\[
A=\begin{bmatrix}
I_{2} & O\\
O & A_{1}
\end{bmatrix},~~~B=\begin{bmatrix}
0 & 2 & 0 & \cdots & 0\\
0 & 0 & 0 & \cdots & 0\\
 & O & & B_{1} & 
\end{bmatrix} \in M_{n}(\mathbb{R}_{+})
\]
for some  $A_{1},  B_{1}\in M_{n-2}(\mathbb{R}_{+})$ with $\Vert A_{1}\Vert\leq 1$ and $\Vert B_{1}\Vert \leq 1$ \cite{CGLTW}. So one has
\[
A\otimes B=\begin{bmatrix}
0 & 2 & 0 & \cdots & 0\\
0 & 0 & 0 & \cdots & 0\\
 & O & & A_{1}\otimes B_{1} & 
\end{bmatrix},~~ B\otimes A=\begin{bmatrix}
0 & 2 & 0 & \cdots & 0\\
0 & 0 & 0 & \cdots & 0\\
 & O & & B_{1}\otimes A_{1} & 
\end{bmatrix}
\]
and hence
\[
W_{max}(A\otimes B)=W_{max}(B\otimes A)=[0, 2].
\]
\end{example}

\section{Joint $k-$numerical range}

 Let $\mathbb{A}=(A_{1}, \ldots, A_{m})$ be an $m-\rm{tuple}$ of $n \times n$  complex matrices. For $1\leq k \leq n,$ the max joint $k-$numerical range of $\mathbb{A}$ is defined as 
\[
W_{k}(\mathbb{A})=\{\displaystyle(tr(X^{*}A_{1}X), \ldots, tr(X^{*}A_{m}X)): X\in \widetilde{\mathcal{X}}_{n\times k} \}.
\]
When $k=1,$ it reduces to the usual joint numerical range of $\mathbb{A}$. It is known that $W_{1}(\mathbb{A})$ is compact and connected. However,  $W_{1}(\mathbb{A})$ is not always convex  if $m>3$   (e.g., see \cite{AYP1979}).   We are interested in studying the  joint  $k-$numerical range  in  the max algebra setting.
\begin{definition}\label{eq1}
Let $\mathbb{A}=(A_1,\dots ,A_m)$, where $A_i\in M_n(\mathbb{R}_+),$  $i=1,\dots ,m$ and  let $1\leq k\leq n$ be a positive integer. The max joint $k$-numerical range of $\mathbb{A}$ is defined and denoted by 
\begin{eqnarray*}
W_{\max}^k(\mathbb{A})&=&\{(tr_{\otimes}(X^t\otimes A_1\otimes X),\dots ,tr_{\otimes}(X^t\otimes A_m\otimes X)):\ X\in\mathcal{X}_{n\times k}\}\\
&=&\{(\displaystyle\oplus_{i=1}^kx_i^t\otimes A_1\otimes x_i,\dots ,\oplus_{i=1}^kx_i^t\otimes A_m\otimes x_i):\ X=[x_1,\dots ,x_k]\in\\
&&\mathcal{X}_{n\times k}\}.
\end{eqnarray*}
\end{definition}
It is clear that $W_{\max}^1(\mathbb{A})=W_{\max}(\mathbb{A})$. So, the notion of max joint $k$-numerical range is a generalization of the max joint numerical range (\cite{TZPA, Erratum}).
\begin{example}\label{ex1}
Let $\mathbb{A}=(A_1,A_2)$, where $A_1=diag(2,4,5),A_2=diag(7,3,5)$. Then $W_{\max}(\mathbb{A})$ and $W_{\max}^2(\mathbb{A})$ are  displayed in Figure \ref{fig1}. It follows $W_{\max}^3(\mathbb{A})=\{(5,7)\}$.
\end{example}

\begin{figure}[!ht]
\begin{center}
\subfloat[$W_{\max}(\mathbb{A})$]{\label{figd1}\includegraphics[scale=0.35]{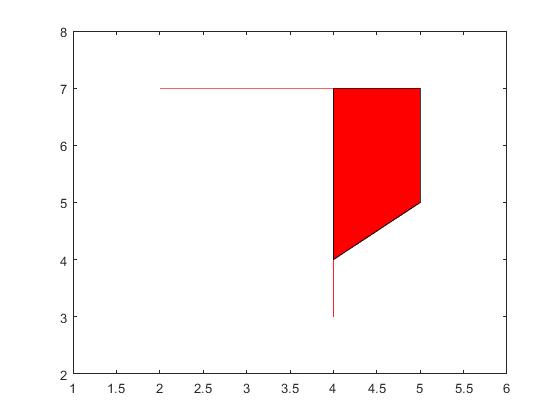}}
\subfloat[$W_{\max}^2(\mathbb{A})$]{\label{figd2}\includegraphics[scale=0.35]{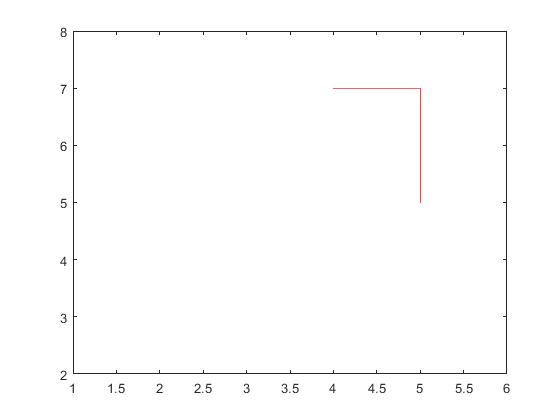}}
\caption{The max joint $k$-numerical range of  $\mathbb{A}=(A_{1}, A_{2}),   W_{\max}^{k}(\mathbb{A}), k=1, 2$  }
\label{fig1}
\end{center}
\end{figure}

The following result extends \cite[Proposition 3.9 and Corollary 3.10]{Erratum} to $m-\rm{tuples}$ of $n \times n$ nonnegative matrices.
\begin{proposition}\label{InequalityLipschitzcontinuousA2}
 Let $\mathbb{A}=(A_1,\dots ,A_m)$, where $A_i\in M_n(\mathbb{R}_+),$  $i=1,\dots ,m$ and let $1\leq k\leq n$ be a positive integer. Consider the map  $f_{\mathbb{A}}:\mathcal{X}_{n\times k}\rightarrow\mathbb{R}_+^m,$ where
\begin{eqnarray*}
 f_{\mathbb{A}}(X)=\displaystyle(tr_{\otimes}(X^t\otimes A_1\otimes X),\dots ,tr_{\otimes}(X^t\otimes A_m\otimes X)).
\end{eqnarray*}
Then
\begin{eqnarray}
\nonumber
\|  f_{\mathbb{A}}(X) - f_{\mathbb{A}}(Y) \| 
 \leq   \Vert \bigoplus_{i=1}^{m}A_{i} \Vert (\|X\| + \|Y\|) \|X-Y\|
 \label{good_L}
\end{eqnarray}
for all $X, Y \in \mathcal{X}_{n\times k}$.
\end{proposition}
\begin{proof} For each $X, Y\in \mathcal{X}_{n\times k}$  there exists $1\leq i_{0}\leq m$  such that
\begin{eqnarray*}
\Vert f_{\mathbb{A}}(X)-f_{\mathbb{A}}(Y) \Vert &=& \vert tr_{\otimes}(X^{t}\otimes A_{i_{0}}\otimes X)- tr_{\otimes}(Y^{t}\otimes A_{i_{0}}\otimes Y)\vert \\
&\leq& \Vert A_{i_0} \Vert (\|X\| + \|Y\|) \|X-Y\|\\
&\leq & \Vert \bigoplus_{i=1}^{m}A_{i} \Vert (\|X\| + \|Y\|) \|X-Y\|,
\end{eqnarray*}
where the first inequality follows from \cite[Proposition 3.9]{Erratum}.
\end{proof}
\begin{corollary} Let $\mathbb{A}=(A_1,\dots ,A_m)$, where $A_i\in M_n(\mathbb{R}_+),$  $i=1,\dots ,m$ and let $1\leq k\leq n$ be a positive integer. 
 For each $Z\in \mathcal{X}_{n\times k}$ and $X,Y\in \mathcal{X}_{n\times k}$ such that 
 $\|X-Z\|\le \frac{1}{2}$ and $\|Y-Z\|\le \frac{1}{2}$
  we have 
 \[
\|  f_{\mathbb{A}}(X)- f_{\mathbb{A}}(Y) \| \leq 
  \Vert \bigoplus_{i=1}^{m}A_{i} \Vert \displaystyle \left(
2\Vert Z\Vert+1
\right)\Vert X-Y\Vert.
 \] 
 Therefore the map $f_{\mathbb{A}}:\mathcal{X}_{n\times k}\rightarrow\mathbb{R}_+^m$  is locally Lipschitz continuous. 
\end{corollary}

%
%
 \begin{remark}
 Note that $W_{\max}^k(\mathbb{A})$ is the image of  continuous mapping $f_{\mathbb{A}}$.  Using  compactness of  $\mathcal{X}_{n\times k}$ (\cite[Remark 3.5]{Erratum}), $W_{\max}^k(\mathbb{A})$ is a   compact set. However, we are unsure of the connectedness of  $W_{\max}^k(\mathbb{A})$, which we pose as an open question (see also \cite[Question 3.2]{Erratum})).
\end{remark}

\begin{proposition}\label{cor1}

Let $\mathbb{A}=(A_1,\dots ,A_m),$ where  $A_l=(a_{ij}^{(l)})\in M_n(\mathbb{R}_+),\ l=1,\dots , m$ and let $1\leq k<n$ be a positive integer. Then 
\begin{equation}
W_{\max}^k(\mathbb{A})\subseteq [s_1,t_1]\times\dots\times [s_m,t_m]\subseteq\mathbb{R}_+^m,
\label{incl1}
\end{equation}
where $s_l=\min\{\displaystyle\oplus_{j=1}^ka_{i_ji_j}^{(l)}:\ 1\leq i_1<i_2<\dots <i_k\leq n\}, t_l=\displaystyle\max_{1\leq i,j\leq n}a_{ij}^{(l)}$ for all $l=1,\dots ,m$. Moreover, when $k=n$, we have
\begin{equation}
W_{\max}^n(\mathbb{A})=\{(tr_{\otimes}(A_1),\dots ,tr_{\otimes}(A_m))\}.
\label{eq_n}
\end{equation}
\end{proposition}

\begin{proof}
Let $\lambda =(\lambda_1,\dots ,\lambda_m)\in W_{\max}^k(\mathbb{A})$ for some $1\leq k<n$. Then there exists $X\in \mathcal{X}_{n\times k}$ such that $\lambda =(tr_{\otimes}(X^t\otimes A_1\otimes X),\dots ,tr_{\otimes}(X^t\otimes A_m\otimes X))$. Therefore, $\lambda_i=tr_{\otimes}(X^t\otimes A_i\otimes X)$ for $i=1,\dots ,m$. So, $\lambda_i\in W_{\max}^k(A_i)$ for $i=1,\dots ,m$. By Theorem \ref{th1}, $\lambda_i\in [s_i,t_i]$ for $i=1,\dots ,m$. Thus
\begin{equation*}
\lambda =(\lambda_1,\dots ,\lambda_m)\in [s_1,t_1]\times\dots \times [s_m,t_m],
\end{equation*}
which establishes (\ref{incl1}).

To prove (\ref{eq_n}) observe that $tr_{\otimes}(X^t\otimes A_i\otimes X)=tr_{\otimes}(A_i)$ for each $X\in\mathcal{U}_n$ and each $i=1,\dots , m$. So we have 
\begin{eqnarray*}
W_{\max}^n(\mathbb{A})&=&\{(tr_{\otimes}(X^t\otimes A_1\otimes X),\dots ,tr_{\otimes}(X^t\otimes A_m\otimes X)):X\in\mathcal{U}_n\}\\
&=&\{(tr_{\otimes}(A_1),\dots ,tr_{\otimes}(A_m))\}.
\end{eqnarray*}
Hence, the proof is complete.
\end{proof}

In the following theorem, we state some basic properties of the max joint $k$-numerical range of matrices.
\begin{theorem}\label{TheoremJointK}
Let $\mathbb{A}=(A_1,\dots ,A_m)$, where $A_i\in M_n(\mathbb{R}_+)$ for all $i=1,\dots ,m$, and let $1\leq k\leq n$ be a positive integer. Then the following assertions are true:
\begin{itemize}
\item[(i)] If $A_1=\dots =A_m=(a_{ij})$. If $1\leq k<n$ be a positive integer, then $W_{\max}^k(\mathbb{A})=\{(a,\dots ,a):\ s\leq a\leq t\}$, where $s=\min\{\displaystyle\oplus_{j=1}^ka_{i_ji_j}:\ 1\leq i_1<i_2<\dots <i_k\leq n\}$ and $t=\displaystyle\max_{1\leq i,j\leq n}a_{ij}$. If $k=n$, then $W_{\max}^n(\mathbb{A})=\{tr_{\otimes}(A),\dots ,tr_{\otimes}(A)\}$;

\item[(ii)] $W_{\max}^k(A_1\oplus\alpha_1I,\dots ,A_m\oplus\alpha_mI)=W_{\max}^k(A_1,\dots ,A_m)\oplus \{(\alpha_1,\dots ,\alpha_m)\}$ and $W_{\max}^k(\mathbb{A}\oplus \mathbb{B})\subseteq W_{\max}^k(\mathbb{A})\oplus W_{\max}^k(\mathbb{B})$, where $\alpha_1,\dots ,\alpha_m\in\mathbb{R}_+$ and $\mathbb{B}=(B_1,\dots ,B_m)$ with $B_i\in M_n(\mathbb{R}_+)$ for all $i=1,\dots ,m$;
\item[(iii)] $W_{\max}^k(U^t\otimes A_1\otimes U,\dots ,U^t\otimes A_m\otimes U)=W_{\max}^k(A_1,\dots ,A_m)$, where $U\in\mathcal{U}_n$;
\item[(iv)] $W_{\max}^n(\alpha_1A_1,\dots ,\alpha_mA_m)=\{(\alpha_1tr_{\otimes}(A_1),\dots ,\alpha_mtr_{\otimes}(A_m))\}$, where $\alpha_1,\dots ,\\ \alpha_m\in\mathbb{R}_+$;
\item[(v)] If $B_1,\dots ,B_m\in M_l(\mathbb{R}_+)$ are formed by considering rows and columns $i_1,\dots ,i_l$ from $A_1,\dots ,A_m$, respectively and $k\leq l$, then $W_{\max}^k(\mathbb{B})\subseteq W_{\max}^k(\mathbb{A})$, where $\mathbb{B}=(B_1,\dots ,B_m)$. Consequently, if $V=[e_{i_1},\dots ,e_{i_s}]\in M_{n\times s}(\mathbb{R}_+),\\ 1\leq k \leq s\leq n,$ then $W_{\max}^k(V^t\otimes A_1\otimes V,\dots ,V^t\otimes A_m\otimes V)\subseteq W_{\max}^k(A_1,\dots ,A_m)$ and the equality holds if $s=n$;
\item[(vi)] Let $B_1,\dots ,B_m\in M_{p}(\mathbb{R}_+)$ and let  $D_i=\left[\begin{array}{cc}
A_i & 0\\
0 & B_i
\end{array}\right]$. Moreover, let $1\leq k\leq\min\{n, p\}$ be a positive integer. Then $W_{\max}^k(\mathbb{A})\cup W_{\max}^k(\mathbb{B})\subseteq W_{\max}^k(\mathbb{D})$, where $\mathbb{B}=(B_1,\dots ,B_m)$ and $\mathbb{D}=(D_1,\dots ,D_m)$;
\item[(vii)] If $m-1<k<n $, then $W_{\max}^{k+1}(\mathbb{A})\subseteq W_{\max}^k(\mathbb{A})$.
\end{itemize}
\end{theorem}

\begin{proof}
To prove (i), let  $A_1=\dots =A_m=(a_{ij})\in M_n(\mathbb{R}_+)$ be given and let $1\leq k<n$ be a positive integer. Then by Definition \ref{eq1} and Theorem \ref{th1}, we have
\begin{eqnarray*}
W_{\max}^k(\mathbb{A})&=&\{(tr_{\otimes}(X^t\otimes A\otimes X),\dots ,tr_{\otimes}(X^t\otimes A\otimes X)):X\in\mathcal{X}_{n\times k}\}\\
&=&\{(a,\dots ,a):s\leq a\leq t\}.
\end{eqnarray*}
The case  $k=n$ follows from (\ref{eq_n}).


For (ii), let $\alpha_1,\dots ,\alpha_m\in\mathbb{R}_+$ be given. Then
\begin{eqnarray*}
&&W_{\max}^k(A_1\oplus\alpha_1I,\dots ,A_m\oplus\alpha_mI)\\
&=&\{(tr_{\otimes}(X^t\otimes (A_1\oplus\alpha_1I)\otimes X),\dots ,tr_{\otimes}(X^t\otimes (A_m\oplus\alpha_mI)\otimes X):X\in\mathcal{X}_{n\times k}\}\\
&=&\{(tr_{\otimes}(X^t\otimes A_1\otimes X)\oplus tr_{\otimes}(X^t\otimes\alpha_1I\otimes X),\dots ,\\
&&tr_{\otimes}(X^t\otimes A_m\otimes X)\oplus tr_{\otimes}(X^t\otimes\alpha_mI\otimes X)):X\in\mathcal{X}_{n\times k}\}\\
&=&\{(tr_{\otimes}(X^t\otimes A_1\otimes X),\dots ,tr_{\otimes}(X^t\otimes A_m\otimes X)):X\in\mathcal{X}_{n\times k}\}\oplus\{(\alpha_1,\dots ,\alpha_m)\}\\
&=&W_{\max}^k(A_1,\dots ,A_m)\oplus \{(\alpha_1,\dots ,\alpha_m)\}.
\end{eqnarray*}

For the second assertion, let $\mathbb{B}=(B_1,\dots ,B_m)$, where $B_1,\dots ,B_m\in M_n(\mathbb{R}_+)$. Then by Definition \ref{eq1}, we have
\begin{eqnarray*}
&&W_{\max}^k(\mathbb{A}\oplus \mathbb{B})\\
&=&\{(tr_{\otimes}(X^t\otimes (A_1\oplus B_1)\otimes X),\dots ,tr_{\otimes}(X^t\otimes (A_m\oplus B_m)\otimes X)):X\in\mathcal{X}_{n\times k}\}\\
&=&\{(tr_{\otimes}(X^t\otimes A_1\otimes X),\dots ,tr_{\otimes}(X^t\otimes A_m\otimes X))\\
&&\oplus (tr_{\otimes}(X^t\otimes B_1\otimes X),\dots ,tr_{\otimes}(X^t\otimes B_m\otimes X)):X\in\mathcal{X}_{n\times k}\}\\
&\subseteq&\{(tr_{\otimes}(X^t\otimes A_1\otimes X),\dots ,tr_{\otimes}(X^t\otimes A_m\otimes X)):X\in\mathcal{X}_{n\times k}\}\\
&&\oplus\{(tr_{\otimes}(X^t\otimes B_1\otimes X),\dots ,tr_{\otimes}(X^t\otimes B_m\otimes X)):X\in\mathcal{X}_{n\times k}\}\\
&=&W_{\max}^k(\mathbb{A})\oplus W_{\max}^k(\mathbb{B}).
\end{eqnarray*}

For (iii), let $U\in\mathcal{U}_n$ be given. By putting $Y=U\otimes X$, where $X\in\mathcal{X}_{n\times k}$, we have $Y\in\mathcal{X}_{n\times k}$ and
\begin{eqnarray*}
&&W_{\max}^k(U^t\otimes A_1\otimes U,\dots ,U^t\otimes A_m\otimes U)\\
&=&\{(tr_{\otimes}(X^t\otimes U^t\otimes A_1\otimes U\otimes X),\dots ,tr_{\otimes}(X^t\otimes U^t\otimes A_m\otimes U\otimes X)):X\in\mathcal{X}_{n\times k}\}\\
&=&\{(tr_{\otimes}(Y^t\otimes A_1\otimes Y),\dots ,tr_{\otimes}(Y^t\otimes A_m\otimes Y)):Y\in\mathcal{X}_{n\times k}\}\\
&=&W_{\max}(A_1,\dots ,A_m).
\end{eqnarray*}

Property (iv) follows from  (\ref{eq_n}). 


To prove (v), let $(\lambda_1,\dots ,\lambda_m)\in W_{\max}^k(\mathbb{B})$ be given. By using  Definition \ref{eq1}, there is an orthonormal set $\{x_{1}, \ldots, x_{k}\}$ in $\mathbb{R}_{+}^{l}$ (i.e., $\|x_i\|=1$ and $x_i ^T \otimes x_j$=0 for $i\neq j$, $i,j=1, \ldots, k$) such that
\[
\lambda_{j}=\displaystyle\oplus_{i=1}^{k}x_{i}^{t}\otimes B_{j}\otimes x_{i}~ j=1, \ldots, m.
\]
 If we define 
 \[
 y_{ir}=  
  \begin{cases}
   x_{i_{s}r}& i=i_{s}   \\ 
    0  & i\neq i_{s}
\end{cases},   1 \leq i \leq n,  1 \leq r \leq k, 1 \leq s \leq l,
 \]
 then  $\{y_1,\dots ,y_k\},$  where $y_{r}=[y_{1r}, \ldots, y_{rr}, \ldots,  y_{nr}]^{t}\in \mathbb{R}_{+}^{n}$, is an orthonormal set in $\mathbb{R}_+^n$ and 
 \[
 x_{r}^{t}\otimes B_{i}\otimes x_{r}=y_{r}^{t}\otimes A_{i}\otimes y_{r}, ~1 \leq r \leq k, 1 \leq i \leq m.
 \]
 Therefore
\[ 
  \lambda_j=\displaystyle\oplus_{i=1}^ky_i^t\otimes A_j\otimes y_i~ \mathrm{for~ all}~ j=1,\ldots ,m 
  \]
 and thus  $(\lambda_1,\dots ,\lambda_m)\in W_{\max}^k(\mathbb{A}),$ which completes the proof.

The assertion (vi) follows from (v).

For (vii), let $(\lambda_1,\dots ,\lambda_m)\in W_{\max}^{k+1}(\mathbb{A})$ be given. By using  Definition \ref{eq1}, there exists $X=[x_1,\dots ,x_{k+1}]\in\mathcal{X}_{n\times (k+1)}$ such that
 \[
\lambda_j=\oplus_{i=1}^{k+1}x_i^t\otimes A_j\otimes x_i~  \mathrm{for~ all} ~j=1,\dots ,m. 
\]
Since $k+1>m$, there is a $1\leq s\leq k+1$ such that 
\[
x_s^t\otimes A_j\otimes x_s\leq \oplus_{i=1,i\neq s}^{k+1}x_i^t\otimes A_j\otimes x_i\ \forall j=1,\dots ,m.
\]
 Now, by setting $X=[x_1,\dots ,x_{s-1},x_{s+1},\dots ,x_{k+1}]$, we have $X\in\mathcal{X}_{n\times k}$ and 
 \[
 \lambda_j=\oplus_{i=1\\ i\neq j}^{k+1}x_i^t\otimes A_j\otimes x_i~ \mathrm{for~ all} ~ j=1,\dots ,m. 
 \]
This implies that  $(\lambda_1,\dots ,\lambda_m)\in W_{\max}^k(\mathbb{A})$ and the proof is complete.
\end{proof}

\section{Joint $C-$numerical range in max algebra}

Let $ A \in M_{n}(\mathbb{R}_{+})$ and $c=[c_{1}, c_{2}, \ldots, c_{n}]^{t} \in \mathbb{R}_{+}^{n}.$ As stated 
in \cite[(4.3)]{Erratum}, the {\em max $c-$numerical range of $A$} as defined in (\ref{c1}) equals
\begin{equation}
W_{\max}^c(A)=\{\displaystyle\oplus_{i=1}^nc_ia_{\sigma(i),\sigma (i)}:\sigma\in \sigma_n\},
\label{tretja}
\end{equation}
and its max convex hull equals
 \[
conv_{\otimes}(W_{\max}^c(A))=[\displaystyle\min_{\sigma\in S_n}\displaystyle\oplus_{i=1}^nc_ia_{\sigma(i),\sigma (i)},  \displaystyle\max_{\sigma\in S_n}\displaystyle\oplus_{i=1}^nc_ia_{\sigma(i),\sigma (i)}]. 
 \]

Next we introduce and study the notion of max joint $c$-numerical range of $m-$tuple of $n\times n$ nonnegative matrices, where $c\in \mathbb{R}_{+}^{n}$.
\begin{definition}\label{def2}
Let $\mathbb{A}=(A_1,\dots ,A_m)$, where $A_i\in M_n(\mathbb{R}_+)$ for all $i=1,\dots ,m$ and let $c=[c_1,\dots ,c_n]^t\in\mathbb{R}_+^n$. The max joint $c$-numerical range of $\mathbb{A}$ is defined and denoted by
\begin{eqnarray*}
W_{\max}^c(\mathbb{A})&=&\{(\oplus_{i=1}^nc_ix_i^t\otimes A_1\otimes x_i,\dots ,\oplus_{i=1}^nc_ix_i^t\otimes A_m\otimes x_i):\\
&& \hspace{2cm} X=[x_1,\dots ,x_n]\in \mathcal{U}_n\}.
\end{eqnarray*}
\end{definition}
It is easy to check that if $c_1=\dots =c_n,$  then 
\begin{equation*}
W_{\max}^c(\mathbb{A})=\displaystyle\{c_{1}(tr_{\otimes}(A_1),\dots ,tr_{\otimes}(A_m))\}.
\end{equation*}

Definition  \ref{tretja} implies the following result, since for each $X=[x_1,\dots ,x_n]\in \mathcal{U}_n$ it holds that $x_i =e_{\sigma (i)}$ for some permutation $\sigma$ of $\{1, \ldots , n\}$.

\begin{proposition}
Let $\mathbb{A}=(A_1,\dots ,A_m)$, where $A_i\in M_n(\mathbb{R}_+)$ for all $i=1,\dots ,m$ and let $c=[c_1,\dots ,c_n]^t\in\mathbb{R}_+^n$.  Then 
\[  
W_{\max}^c(\mathbb{A})=\displaystyle\{(\bigoplus_{i=1}^{n}c_{i}(A_{1})_{\sigma(i)\sigma(i)}, \ldots,  \bigoplus_{i=1}^{n}c_{i}(A_{m})_{\sigma(i)\sigma(i)}: \sigma\in S_{n} \}.
\]
\end{proposition}
In the following remark we state the relation between $W_{\max}^c(\mathbb{A})$ and $W_{\max}^c(A_i)$'s,    $1\leq i \leq m$.
\begin{remark}
Let $\mathbb{A}=(A_1,\dots ,A_m),$ where $A_i\in M_n(\mathbb{R}_+), i=1, 2, \ldots, m$ be $m$ nonnegative matrices and let $c=[c_1,\dots ,c_n]^t\in\mathbb{R}_{+}^{n}$. Then
\[
W_{\max}^c(\mathbb{A})\subseteq W_{\max}^c(A_1) \times \cdots \times W_{\max}^c(A_m)
\]
and the equality holds if $W_{\max}^c(A_i)$'s,  $1\leq i \leq m$ are singletons.
\end{remark}
The following remark shows  some  cases in  which $ W_{max}^{c}(\mathbb{A})$ is a singleton set.
\begin{remark}
Let $\mathbb{A}=(A_1,\dots ,A_m),$ where $A_i\in M_n(\mathbb{R}_+), i=1, 2, \ldots, m$,   and let $c=[c_1,\dots ,c_n]^t\in\mathbb{R}_{+}^{n}$.  Then
\begin{itemize}
\item[(i)]
If $\displaystyle(A_{i})_{jj}=\alpha_{i},  1\leq i \leq m,  1\leq j \leq n,$ then
\[
W_{max}^{c}(A_{i})=\{\alpha_{i}\oplus_{j=1}^{n}c_{j} \},  W_{max}^{c}(\mathbb{A})=\{(\alpha_{1}\oplus_{j=1}^{n}c_{j},  \ldots, \alpha_{m}\oplus_{j=1}^{n}c_{j}) \}, 1\leq i \leq m.
\]  
\item[(ii)] If $c_1=\dots =c_n,$  then
\[
W_{\max}^c(A_i)=\{c_{1}\oplus_{j=1}^n(A_i)_{jj}\},W_{\max}^c(\mathbb{A})=\{(c_{1}\oplus_{j=1}^n(A_1)_{jj},\dots ,c_{1}\oplus_{j=1}^n(A_m)_{jj}\}.
\]
\end{itemize}
\end{remark}
\begin{remark}
Let $\mathbb{A}=(A_1,\dots ,A_m),$ where $A_i\in M_n(\mathbb{R}_+), i=1, 2, \ldots, m$   and let $c=[c_1,\dots ,c_n]^t\in\mathbb{R}_{+}^{n}$. Moreover suppose that 
$W_{c}^{max}(A_{i}),  1\leq i \leq m$,  has $r_i$ elements, where  $1\leq r_i\leq n!, i=1, \ldots, m$ and $r=\displaystyle\max_{1\leq i\leq m}r_i=r$. Then the number of elements of $W_{\max}^c(\mathbb{A})$ is between $r$ and $n!$.

In particular, if $\mathbb{A}=(A_1,A_2),$ where $A_{1}, A_{2}$ are  two $2\times 2$ nonnegative  matrices and let $c=[c_{1}, c_{2}]^{t}$. Then each of the sets  $W_{max}^{c}(A_{1}),$   $W_{max}^{c}(A_{2})$ and  $W_{max}^{c}(\mathbb{A})$  has at most two elements.
\end{remark}

\begin{example}
Let $\mathbb{A}=(A_1,A_2)$, where $A_1=
\begin{bmatrix}
5 & 2\\
7 & 4
\end{bmatrix},A_2=\begin{bmatrix}
3 &4\\
2 & 8
\end{bmatrix}$ and $c=[2,8]^t$. So  $W_{\max}^c(A_1)=\{32,40\},W_{\max}^c(A_2)=\{24,64\}$ and
 \[
W_{\max}^c(\mathbb{A})=\{(32,64),(40,24)\}.
\]
\end{example}

%
%
Let $\mathbb{A}=(A, B),$  where $A, B \in M_{2}(\mathbb{R}_{+})$ are two nonnegative matrices and let $W_{max}^{c}(A)=\{a_{1}, a_{2}\}, W_{max}^{c}(B)=\{b_{1}, b_{2}\}$   with $a_{2}>a_{1}$ and $b_{2}>b_{1}$. Does the following equality hold:
$W_{max}^{c}(\mathbb{A})=\{(a_{1}, b_{2}), (a_{2}, b_{1}) \}$?
The next example shows that the answer to this  question is negative in general
\begin{example}
Let $\mathbb{A}=(A_1,A_2)$, where $A_1=
\begin{bmatrix}
5 & 2\\
7 & 4
\end{bmatrix}, A_2=\begin{bmatrix}
9 &4\\
8 & 7
\end{bmatrix}$ and $c=[4, 3]^t$. So  $W_{\max}^c(A_1)=\{20, 16\}, W_{\max}^c(A_2)=\{36, 28\}$ and
 \[
W_{\max}^c(\mathbb{A})=\{(20, 36), (16, 28)\}.
\]
\end{example}


\begin{example}
Let $\mathbb{A}=(A_1,A_2,A_3)$, where $A_1=\begin{bmatrix}
2 &3 &4\\
5 &7 & 1\\
2 & 6 & 8
\end{bmatrix},A_2=\begin{bmatrix} 
3 &4 &5\\
0 & 5 & 2\\
1 & 7 &2
\end{bmatrix},\\A_3=\begin{bmatrix}
4 &3 & 1\\
5 & 6 & 2\\
3 &4 & 6
\end{bmatrix}
$ and $c=[3,4,5]^t$. Then 
\[
W_{\max}^c(A_1)=\{40,35,28,32\},W_{\max}^c(A_2)=\{20,25,15\},W_{\max}^c(A_3)=\{30,24\}
\]
and 
\[
W_{\max}^c(\mathbb{A})=\{(40,20,30),(35,25,30),(40,15,30),(28,20,24),(32,15,24)\}.
\]
\end{example}

By analogy with  \cite[Definition 7]{TZPA} for a single $n\times n$ nonnegative matrix we introduce the following definition for an $m-\rm{tuple}$ of $n \times n$ nonnegative matrices.
\begin{definition}\label{def3}
Let $\mathbb{A}=(A_1,\dots ,A_m)$, where $A_i\in M_n(\mathbb{R}_+)$ for all $i=1,\dots ,m$ and let $C\in M_n(\mathbb{R}_+)$. The max joint $C$-numerical range of $\mathbb{A}$ is defined and denoted by
\begin{equation*}
W_{\max}^C(\mathbb{A})=\displaystyle\{(tr_{\otimes}(C\otimes X^t\otimes A_1\otimes X),\dots , tr_{\otimes}(C\otimes X^t\otimes A_m\otimes X)):\ X\in\mathcal{U}_n\}.
\end{equation*}
\end{definition}

\begin{remark}\label{WcWC}
Let  $\mathbb{A}=(A_1,\dots ,A_m),$ where $A_i\in M_n(\mathbb{R}_+)$ for all $i=1,\dots ,m$ and let  also $c=[c_1,\dots ,c_n]^t\in\mathbb{R}_+^n$. If $C=diag(c_1,\dots ,c_n),$  then one can easily check  that
\begin{equation*}
W_{\max}^c(\mathbb{A})=W_{\max}^C(\mathbb{A}).
\end{equation*}
Thus  this concept is a generalization of the max joint $c$-numerical range.
\end{remark}

\begin{remark}\label{trace}
It is well known that  for $A,B\in M_{n\times n}(\mathbb{R}_+)$ we have $tr_{\otimes}(A^t)=tr_{\otimes}(A)$ and $tr_{\otimes}(A\otimes B)=tr_{\otimes}(B\otimes A)$.
%
\end{remark}

In the following theorem, we state some basic properties of the max joint $C$-numerical range of a tuple of matrices. 

\begin{theorem}\label{th3}
Let $\mathbb{A}=(A_1,\dots ,A_m)$, where $A_i\in M_n(\mathbb{R}_+)$ for all $i=1,\dots ,m$, and $C\in M_n(\mathbb{R}_+)$. Then the following assertions hold:
\begin{itemize}
\item[(i)] $W_{\max}^C(A_1\oplus\alpha_1I,\dots ,A_m\oplus\alpha_mI)=W_{\max}^C(A_1,\dots ,A_m)\oplus \{(\alpha_1,\dots ,\alpha_m)tr_{\otimes}(C)\}$, where $\alpha_1,\dots ,\alpha_m\in\mathbb{R}_+$;
\item[(ii)] $W_{\max}^C(\mathbb{A}\oplus \mathbb{B})\subseteq W_{\max}^C(\mathbb{A})\oplus W_{\max}^C(\mathbb{B})$, and $W_{\max}^{C\oplus D}(\mathbb{A})\subseteq W_{\max}^C(\mathbb{A})\oplus W_{\max}^D(\mathbb{A})$, where $\mathbb{B}=(B_1,\dots ,B_m)$ such that $B_i\in M_n(\mathbb{R}_+)$ for all $i=1,\dots ,m$ and $D\in M_n(\mathbb{R}_+)$;
\item[(iii)] $W_{\max}^C(U^t\otimes A_1\otimes U,\dots ,U^t\otimes A_m\otimes U)=W_{\max}^C(A_1,\dots ,A_m)$, where $U\in\mathcal{U}_n$;
\item[(iv)]  $W_{\max}^C(A_1^t,\dots ,A_m^t)=W_{\max}^{C^t}(A_1,\dots ,A_m)$;
\item[(v)] If $C=\alpha I_n$ for $\alpha\in\mathbb{R}_+$, then $W_{\max}^C(\mathbb{A})=\{\alpha (tr_{\otimes}(A_1),\dots ,tr_{\otimes}(A_m))\}$;
\item[(vi)] $W_{\max}^{(\alpha C\oplus\beta I)}(\mathbb{A})=\alpha W_{\max}^C(\mathbb{A})\oplus\{\beta (tr_{\otimes}(A_1),\dots ,tr_{\otimes}(A_m))\}$, where $\alpha ,\beta\in\mathbb{R}_+$;
\item[(vii)] $W_{\max}^{V^t\otimes C\otimes V}(\mathbb{A})=W_{\max}^C(\mathbb{A})$, where $V\in\mathcal{U}_n$.
\end{itemize}
\end{theorem}

\begin{proof}
To prove (i), let $\alpha_1,\dots ,\alpha_m\in\mathbb{R}_+$ be given. Then by Definition \ref{def3}, we have
\begin{eqnarray*}
&&W_{\max}^C(A_1\oplus\alpha_1I,\dots ,A_m\oplus\alpha_mI)\\
&=&\{(tr_{\otimes}(C\otimes X^t\otimes (A_1\oplus\alpha_1I)\otimes X),\dots ,tr_{\otimes}(C\otimes X^t\otimes (A_m\oplus\alpha_mI)\otimes X)):X\in\mathcal{U}_n\}\\
&=&\{(tr_{\otimes}(C\otimes X^t\otimes A_1\otimes X)\oplus \alpha_1tr_{\otimes}(C),\dots ,tr_{\otimes}(C\otimes X^t\otimes A_m\otimes X)\oplus \alpha_mtr_{\otimes}(C)): X\in\mathcal{U}_n\}\\
&=&\{(tr_{\otimes}(C\otimes X^t\otimes A_1\otimes X),\dots ,tr_{\otimes}(C\otimes X^t\otimes A_m\otimes X))\oplus(\alpha_1,\dots ,\alpha_m)tr_{\otimes}(C):X\in\mathcal{U}_n\}\\
&=&W_{\max}^C(A_1,\dots ,A_m)\oplus \{(\alpha_1,\dots ,\alpha_m)tr_{\otimes}(C)\}.
\end{eqnarray*}

For (ii), let $\mathbb{B}=(B_1,\dots ,B_m)$, where $B_i\in M_n(\mathbb{R}_+)$ for  $i=1,\dots , m$. Then
\begin{eqnarray*}
&&W_{\max}^C(\mathbb{A}\oplus \mathbb{B})\\
&=&\{(tr_{\otimes}(C\otimes X^t\otimes (A_1\oplus B_1)\otimes X),\dots ,tr_{\otimes}(C\otimes X^t\otimes (A_m\oplus B_m)\otimes X)):X\in\mathcal{U}_n\}\\
&=&\{(tr_{\otimes}(C\otimes X^t\otimes A_1\otimes X)\oplus tr_{\otimes} (C\otimes X^t\otimes B_1\otimes X),\dots ,\\
&&tr_{\otimes}(C\otimes X^t\otimes A_m\otimes X)\oplus tr_{\otimes}(C\otimes X^t\otimes B_m\otimes X)):X\in\mathcal{U}_n\}\\
&\subseteq&\{(tr_{\otimes}(C\otimes X^t\otimes A_1\otimes X),\dots ,tr_{\otimes}(C\otimes X^t\otimes A_m\otimes X)):X\in\mathcal{U}_n\}\\
&\oplus&\{(tr_{\otimes}(C\otimes X^t\otimes B_1\otimes X),\dots ,tr_{\otimes}(C\otimes X^t\otimes B_m\otimes X)):X\in\mathcal{U}_n\}\\
&=&W_{\max}^C(\mathbb{A})\oplus W_{\max}^C(\mathbb{B}).
\end{eqnarray*}
 
 For the second assertion in (ii), let $D\in M_n(\mathbb{R}_+)$ be given. Then
 \begin{eqnarray*}
 &&W_{\max}^{C\oplus D}(\mathbb{A})\\
 &=&\{(tr_{\otimes}((C\oplus D)\otimes X^t\otimes A_1\otimes X),\dots ,tr_{\otimes}((C\oplus D)\otimes X^t\otimes A_m\otimes X)):X\in\mathcal{U}_n\}\\
 &=&\{(tr_{\otimes}(C\otimes X^t\otimes A_1\otimes X)\oplus tr_{\otimes}(D\otimes X^t\otimes A_1\otimes X),\dots ,\\
 &&tr_{\otimes}(C\otimes X^t\otimes A_m\otimes X)\oplus tr_{\otimes}(D\otimes X^t\otimes A_m\otimes X)):X\in\mathcal{U}_n\}\\
 &=&\{(tr_{\otimes}(C\otimes X^t\otimes A_1\otimes X),\dots ,tr_{\otimes}(C\otimes X^t\otimes A_m\otimes X))\\
 &\oplus& (tr_{\otimes}(D\otimes X^t\otimes A_1\otimes X),\dots ,tr_{\otimes}(D\otimes X^t\otimes A_m\otimes X)):X\in\mathcal{U}_n\}\\
 &\subseteq&\{(tr_{\otimes}(C\otimes X^t\otimes A_1\otimes X),\dots ,tr_{\otimes}(C\otimes X^t\otimes A_m\otimes X)):X\in\mathcal{U}_n\}\\
 &\oplus&\{ (tr_{\otimes}(D\otimes X^t\otimes A_1\otimes X),\dots ,tr_{\otimes}(D\otimes X^t\otimes A_m\otimes X)):X\in\mathcal{U}_n\}\\
 &=&W_{\max}^{C}(\mathbb{A})\oplus W_{\max}^{ D}(\mathbb{A}).
\end{eqnarray*} 
 
 For (iii), let $U\in\mathcal{U}_n$ be given. If $X\in\mathcal{U}_n$, by setting $Y=U\otimes X$, we have
  \begin{eqnarray*}
&& W_{\max}^C(U^t\otimes A_1\otimes U,\dots ,U^t\otimes A_m\otimes U)\\
 &=&\{(tr_{\otimes}(C\otimes X^t\otimes U^t\otimes A_1\otimes U\otimes X),\dots ,tr_{\otimes}(C\otimes X^t\otimes U^t\otimes A_m\otimes U\otimes X)):X\in\mathcal{U}_n\}\\
 &=&\{(tr_{\otimes}(C\otimes Y^t\otimes A_1\otimes Y),\dots ,tr_{\otimes}(C\otimes Y^t\otimes A_m\otimes Y)):Y\in\mathcal{U}_n\}\\
 &=&W_{\max}^C(A_1,\dots ,A_m).
 \end{eqnarray*} 
 
 (iv) By Definition \ref{def3} and Remark \ref{trace} it follows
 \begin{eqnarray*}
 &&W_{\max}^C(A_1^t,\dots ,A_m^t)\\
 &=&\{(tr_{\otimes}(C\otimes X^t\otimes A_1^t\otimes X),\dots ,tr_{\otimes}(C\otimes X^t\otimes A_m^t\otimes X)):X\in\mathcal{U}_n\}\\
 &=&\{(tr_{\otimes}(X^t\otimes A_1\otimes X\otimes C^t),\dots ,tr_{\otimes}(X^t\otimes  A_m\otimes X\otimes C^t)):X\in\mathcal{U}_n\}\\
 &=&\{(tr_{\otimes}(C^t\otimes X^t\otimes A_1\otimes X),\dots ,tr_{\otimes}(C^t\otimes X^t\otimes A_m\otimes X)):X\in\mathcal{U}_n\}\\
 &=&W_{\max}^{C^t}(A_1,\dots ,A_m).
 \end{eqnarray*} 
 
 To prove (v), let $\alpha\in\mathbb{R}_+)$ be given. Then
 \begin{eqnarray*}
 &&W_{\max}^{\alpha I_n}(A_1,\dots ,\alpha_m)\\
 &=&\{(tr_{\otimes}(\alpha I_n\otimes X^t\otimes A_1\otimes X),\dots ,tr_{\otimes}(\alpha I_n\otimes X^t\otimes A_m\otimes X)):X\in\mathcal{U}_n\}\\
 &=&\{(tr_{\otimes}(\alpha A_1),\dots ,tr_{\otimes}(\alpha A_m)):X\in\mathcal{U}_n\}\\
 &=&\{\alpha (tr_{\otimes}(A_1),\dots ,tr_{\otimes}(A_m))\}.
  \end{eqnarray*}

To prove (vi), let $(\lambda_1,\dots ,\lambda_m)\in W_{\max}^{(\alpha C\oplus\beta I)}(\mathbb{A})$ be given. By Definition \ref{def3}, there exists $X\in\mathcal{U}_n$ such that $\lambda_i=tr_{\otimes}((\alpha C\oplus\beta I)\otimes X^t\otimes A_i\otimes X)$ for all $i=1,\dots ,m$. So, we have
\begin{equation*}
\lambda_i=\alpha tr_{\otimes}(C\otimes X^t\otimes A_i\otimes X)\oplus\beta tr_{\otimes}(A_i)\ \ \forall i=1,\dots ,m.
\end{equation*}
Therefore, $(\lambda_1,\dots ,\lambda_m)\in\alpha W_{\max}^C(\mathbb{A})\oplus\{\beta (tr_{\otimes}(A_1),\dots ,tr(A_m))\}$  and so the inclusion   $W_{\max}^{(\alpha C\oplus\beta I)}(\mathbb{A})\subseteq\alpha W_{\max}^C(\mathbb{A})\oplus \{\beta (tr_{\otimes}(A_1),\dots ,tr_{\otimes}(A_m))\}$ follows. The reverse inclusion can  be verified in a similar manner.

For (vii), let $(\lambda_1,\dots ,\lambda_m)\in W_{\max}^{V^t\otimes C\otimes V}(\mathbb{A})$ be given. Then by Definition \ref{def3}, there exists $X\in\mathcal{U}_n$ such that $\lambda_i=tr_{\otimes}(V^t\otimes C\otimes V\otimes X^t\otimes A_i\otimes X)$ for all $i=1,\dots ,m$. Now, setting $Y=X\otimes V^t\in\mathcal{U}_n$, for all $i=1,\dots ,m$ we have 
\begin{eqnarray*}
\lambda_i &=&tr_{\otimes}(V^t\otimes C\otimes V\otimes X^t\otimes A_i\otimes X)\\
&=&tr_{\otimes}(C\otimes (X\otimes V^t)^t\otimes A_i\otimes (X\otimes V^t))\\
&=&tr_{\otimes}(C\otimes Y^t\otimes A_i\otimes Y).
\end{eqnarray*}
Thus $(\lambda_1,\dots ,\lambda_m)\in W_{\max}^C(\mathbb{A})$. Therefore, $W_{\max}^{V^t\otimes C\otimes V}(\mathbb{A})\subseteq W_{\max}^C(\mathbb{A})$. 
Since by the already established inclusion (applied to $V^t$) we have
$$W_{\max}^C(\mathbb{A})= W_{\max}^{V\otimes V^t\otimes C \otimes V\otimes V^t}(\mathbb{A})\subseteq W_{\max}^{V^t\otimes C \otimes V}(\mathbb{A}),$$
the proof is complete.
\end{proof}
By Theorem \ref{th3} and Remark \ref{WcWC} we have the following consequence.
\begin{corollary}
Let $\mathbb{A}=(A_1,\dots ,A_m)$, where $A_i\in M_n(\mathbb{R}_+)$ for all $i=1,\dots ,m$, and $c=[c_1,\dots ,c_n]^t\in\mathbb{R}_+^n$. Then the following assertions hold:
\begin{itemize}
\item[(i)] $W_{\max}^c(A_1\oplus\alpha_1I,\dots ,A_m\oplus\alpha_mI)=W_{\max}^c(A_1,\dots ,A_m)\oplus\{ (\alpha_1,\dots ,\alpha_m)(\oplus_{i=1}^nc_i)\}$, where $\alpha_1,\dots ,\alpha_m\in\mathbb{R}_+$;
\item[(ii)] $W_{\max}^c(\mathbb{A}\oplus \mathbb{B})\subseteq W_{\max}^c(\mathbb{A})\oplus W_{\max}^c(\mathbb{B})$ and $W_{\max}^{c\oplus d}(\mathbb{A})\subseteq W_{\max}^c(\mathbb{A})\oplus W_{\max}^d(\mathbb{A})$, where $\mathbb{B}=(B_1,\dots ,B_m)$ such that $B_i\in M_n(\mathbb{R}_+)$ for all $i=1,\dots ,m$ and $d=[d_1,\dots ,d_n]^t\in\mathbb{R}_+^n$;
\item[(iii)] $W_{\max}^c(U^t\otimes A_1\otimes U,\dots ,U^t\otimes A_m\otimes U)=W_{\max}^c(A_1,\dots ,A_m)$, where $U\in\mathcal{U}_n$;
\item[(iv)] $W_{\max}^c(A_1^t,\dots ,A_m^t)=W_{\max}^c(A_1,\dots ,A_m)$;
\item[(v)] If $c_1=\dots =c_n$, then  $W_{\max}^c(\mathbb{A})=\{c_1 (tr_{\otimes}(A_1),\dots ,tr_{\otimes}(A_m))\}$.
\end{itemize}
\end{corollary}

\section*{Acknowledgments}

We thank S. Gaubert for pointing out and correcting some mistakes in the early version of this article. We also thank R. Tayebi Khorami for his consulting activities.

The third author acknowledges a partial support of the Sloveniah Research Agency (grants P1-0222, J1-8133 and J2-2512). Fallat's research is supported in part by an NSERC Discovery Grant, RGPIN--2019--03934.

\bibliographystyle{amsplain}

\end{document}